\documentclass{article}[12pt]
\usepackage{amsmath,amsfonts,amsthm}
\usepackage{amssymb}

\usepackage{epsfig}

\usepackage{subfigure}
\usepackage{graphics}
\usepackage{graphicx}
\usepackage{color}
\usepackage{soul}
\usepackage[makeroom]{cancel}

\usepackage{latexsym}
\usepackage{graphics}
\usepackage{amsmath}
\usepackage{amsthm}
\usepackage{xspace}
\usepackage{amssymb}
\usepackage{color}
\usepackage{cite}
\usepackage{latexsym}
\usepackage{graphics}
\usepackage{amsmath}
\usepackage{amsthm}
\usepackage{xspace}
\usepackage{amssymb}
\usepackage{epsfig}
\usepackage{comment}

\newcommand{\beql}[1]{\begin{equation}\label{#1}}
\newcommand{\eeql}{\end{equation}}
\newcommand{\eqn}[1]{(\ref{#1})}

\newcommand{\R}{\mathbb{R}}
\newcommand{\pr}{\mathbb{P}}
\newcommand{\E}{\mathbb{E}}

\newcommand{\ci}{{\cal I}}

\newcommand{\ck}{{\cal K}}
\newcommand{\cx}{{\cal X}}
\newcommand{\cm}{{\cal M}}

\newcommand{\ch}{{\cal H}}

\newcommand{\ct}{{\cal T}}

\newcommand{\bk}{\boldsymbol{k}}
\newcommand{\bx}{\boldsymbol{x}}
\newcommand{\by}{\boldsymbol{y}}
\newcommand{\bX}{\boldsymbol{X}}

\newcommand{\bbeta}{\boldsymbol{\eta}}
\newcommand{\bnu}{\boldsymbol{\nu}}
\newcommand{\be}{\boldsymbol{e}}
\newcommand{\bu}{\boldsymbol{u}}
\newcommand{\bZero}{\boldsymbol{0}}

\newcommand{\bv}{\boldsymbol{v}}

\newcommand{\Z}{\mathbb{Z}}

\newcommand{\veps}{\varepsilon}
\newcommand{\wt}{\widetilde}

\newtheorem{thm}{Theorem}
\newtheorem{lem}[thm]{Lemma}
\newtheorem{prop}[thm]{Proposition}
\newtheorem{cor}[thm]{Corollary}

\newtheorem{definition}[thm]{Definition}

\begin{document}

\title{A service system with packing constraints: Greedy randomized algorithm achieving sublinear 
 in scale 
optimality gap
}

\author
{
Alexander L. Stolyar \\
University of Illinois Urbana-Champaign\\
1308 W. Main Street, 156 CSL\\
Urbana, IL 61801\\
\texttt{stolyar@illinois.edu}
\and
Yuan Zhong\\
University of Chicago, Booth School of Business\\
5807 S. Woodlawn Ave, HCC 360\\
Chicago, IL 60637\\
\texttt{yuan.zhong@chicagobooth.edu}
}

\date{\today}

\maketitle

\begin{abstract}

A service system with multiple types of arriving customers is considered. 
There is an infinite number of homogeneous servers. 
Multiple customers can be placed for simultaneous service into 
one server, subject to 
general {\em packing constraints}.
The service times of different customers are independent, even if they are served simultaneously
by the same server; the service time distribution depends on the customer type.
Each new arriving customer is placed for service immediately, either
 into an occupied server, i.e., one already serving other customers, 
as long as packing constraints are not violated,
or into an empty server.
After service completion, each customer leaves its server and the system.
The basic objective is to minimize the number of occupied servers in steady state.

We study a {\em Greedy-Random} (GRAND) placement (packing) algorithm, introduced in  \cite{StZh2013}.
This is a simple online algorithm, which places  each arriving customer  uniformly at random into either one of the
already occupied servers that can still fit the customer, 
or one of the so-called {\em zero-servers}, which are empty servers designated to be available to new arrivals.
In \cite{StZh2013}, a version of the algorithm, labeled GRAND($aZ$), was considered,  
where the number of zero servers is $aZ$, with $Z$ being the
current total number of customers in the system, and $a>0$ being an algorithm parameter. 
GRAND($aZ$) was shown in \cite{StZh2013} to be asymptotically optimal in the 
following sense: (a)
the steady-state optimality gap grows 
linearly in the system scale $r$ (the mean total number 
of customers in service), i.e. as $c(a) r$ for some $c(a)> 0$; and (b) $c(a) \to 0$ as $a\to 0$. 

In this paper, we consider the GRAND($Z^p$)
algorithm, in which the number of zero-servers is $Z^p$, where $p \in (1-1/(8\kappa),1)$ is an algorithm parameter,  
and $(\kappa-1)$ is the maximum possible number of customers that a server can fit.
We prove the asymptotic optimality of GRAND($Z^p$) in the sense that the steady-state optimality gap
is $o(r)$, sublinear in the system scale. This is a stronger form of asymptotic optimality than that of GRAND($aZ$).

\end{abstract}

\newpage

\section{Introduction}
\label{sec-intro}

Efficient resource allocation in 
modern cloud computing systems poses many interesting and challenging new problems; see e.g., \cite{Gulati2012} for an overview. 
One such problem is that of efficient real-time assignment (or, packing) of virtual machines to physical machines in a cloud  data center, 
where a primary objective 
is to minimize the number of physical machines being used.
This leads to stochastic dynamic bin packing models, 
where, in contrast to many classical bin packing models,  
``items'' (virtual machines) being placed into ``bins'' (physical machines) do not stay in the system forever,
but leave after a random ``service time.''  For this reason, such models are 
naturally viewed as service (or queueing) systems, with ``items'' and ``bins'' being viewed as customers and servers, respectively.

In this paper, we consider a 
general service system model 
that has also been studied in \cite{St2012,StZh2012,StZh2013}.
There is a finite number of 
customer types, 
and the number of available servers is infinite. 
Customers arrive to the system over time, and 
multiple customers can be placed for simultaneous service (or {\em fit}) into 
the same server, subject to 
{\em packing constraints}. 
We consider {\em monotone} packing, 
a general class of packing constraints, 
where we only impose the following natural and non-restrictive {\em monotonicity} condition: 
if a certain set of customers can fit 
into a server, then a subset can fit as well.
The servers are {\em homogeneous} in that they all have the same packing constraints.
The service times of different customers are independent, even if they are served simultaneously 
by the same server, 
and the service time distribution depends only on the customer type.
Each new arriving customer is placed for service immediately, either
 into an occupied server, i.e., one that is already serving other customers, 
as long as packing constraints are not violated,
or into an empty server.
After service completion, each customer leaves its server and the system -- as we mentioned,
this is what distinguishes 
our model from many classical 
bin packing models (see, e.g., \cite{Csirik2006,Bansal2009}).
As in \cite{StZh2013}, we make Markov assumptions, 
where customers of each type arrive as an independent Poisson process, 
and service time distributions are all exponential.

We are interested in designing
customer placement (packing) algorithms 
that minimize the total number of occupied servers in steady state.
In addition, given the scale at which modern cloud data centers operate, 
it is highly desirable that a placement algorithm is {\em online}, i.e., it makes decisions based
on the current system state only, 
and {\em parsimonious}, 
i.e., it requires only minimal knowledge of system structure and state information.

We study a {\em Greedy-Random} (GRAND) placement algorithm, introduced in  \cite{StZh2013}.
This is a very parsimonious 
online algorithm, which places  each arriving customer  uniformly at random into either one of the
already occupied servers (subject to packing constraints),
or one of the so-called {\em zero-servers}, which are empty servers designated to be available to new arrivals.
In \cite{StZh2013}, a version of the algorithm, 
which we call GRAND($aZ$), 
was considered, where 
the number of zero servers 
depend on the current system state as $aZ$, 
with $Z$ being the current total 
number of customers in the system, and $a>0$ being an algorithm parameter. 
GRAND($aZ$) was shown in \cite{StZh2013} to be 
asymptotically optimal, in the 
sense of the following two properties:
(a) for each $a > 0$, the steady-state optimality gap is of the form $c^r(a) r$, where $r$ is the system {\em scale}, 
defined to be the expectation of $Z$ in steady state, with $c^r(a) \to c(a) > 0$ as $r \to \infty$; 
and (b)  $c(a) \to 0$ as $a\to 0$.
In other words, under GRAND($aZ$) the optimality gap grows linearly as the system scale $r$, 
with the linear factor going to $0$ as the algorithm parameter $a \to 0$.

The focus of this paper is the GRAND($Z^p$)
algorithm, in which the number of zero-servers is $Z^p$, where $p \in \left(1 - \frac{1}{8\kappa}, 1\right)$ is an algorithm
parameter, and $(\kappa-1)$ is the maximum possible number of customers that a server can fit.
Our {\bf main result} is the asymptotic optimality of GRAND($Z^p$), in the sense that the steady-state optimality gap
is $o(r)$, i.e. it is sublinear in the system scale. 
This is a stronger form of asymptotic optimality than that of GRAND($aZ$), because  GRAND($Z^p$) achieves the sublinear gap
simply as $r\to\infty$, 
without having to take an additional limit on 
any algorithm parameter. 
This is in contrast to the case of GRAND($aZ$), which achieves asymptotic optimality 
only by first taking the limit $r \to \infty$, and then the limit $a \to 0$ on the algorithm parameter $a$.

Let us provide some remarks on the reasons why
this stronger form of the asymptotic optimality of GRAND($Z^p$) 
is, on the one hand, natural to
expect, 
and, on the other hand, 
substantially more difficult to rigorously prove. 
Informally speaking, GRAND($Z^p$) can be viewed as GRAND($aZ$), 
where the parameter $a$, 
instead of being fixed, 
is replaced by a variable 
that decreases to zero with increasing system scale; namely, 
$a=Z^p/Z = Z^{p-1}$. However, the asymptotic optimality of GRAND($aZ$) does {\em not}, of course,
imply the (stronger form of) asymptotic optimality of GRAND($Z^p$), 
because the system scale $r\approx Z$, and, 
therefore, the ``parameter'' 
$a=Z^{p-1}\approx r^{p-1}$
depends on and changes with the scale $r$. The key technical difficulty is 
that the analysis of GRAND($Z^p$) 
{\em cannot} be reduced to the analysis of system fluid limits and/or local fluid limits. 
See Section \ref{sec-lyap-choice} for a more detailed discussion.

The implementation of GRAND($Z^p$) is exactly the same as GRAND($aZ$), 
with the only difference being that GRAND($Z^p$) designates $\lceil Z^p \rceil$ 
zero-servers, instead of $\lceil aZ \rceil$, as in the case of GRAND($aZ$).
Thus, like GRAND($aZ$), it is extremely simple and easy to implement. 
It does not need to keep track of the exact states (i.e., ``packing configurations''), of the servers, and the only information required 
at any time is which customer types a given server can still accept for service. 
(For example, for each customer type, a list of servers which can accept 
an additional customer of this type, can be maintained.)
As a result, the algorithm only needs to maintain a very small amount of information.
Furthermore, the algorithm does not use knowledge of 
the customer arrival rates or expected service times. 
We refer the readers to \cite{StZh2013} for a more detailed discussion 
of these attractive implementational features. 

\subsection{Brief Literature Review}

Our work is related to several lines of previous research. 
First, our work is related to the extensive literature on classical bin packing problems, 
where items of various sizes arrive to the system, 
and need to be placed in finite-size bins, according to an online algorithm. 
Once placed, items never leave or move between bins. 
The {\em worst-case analysis} of such problems considers all possible instances of item sizes 
and sequencings, and aims to develop simple algorithms with performance guarantee over all problem instances. 
See e.g., \cite{CCGMV2013} for a recent, extensive survey. 
The {\em stochastic analysis} assumes that item sizes are given according to 
a probability distribution, and the typical objective is to minimize 
the expected number of occupied bins. 
For an overview of results, see e.g., \cite{Csirik2006} and references therein. 
A recent paper \cite{GR2012} establishes 
improved results 
for the classical stochastic setting, and contains
some heuristics and simulations for the case with item departures,
which is a special case of our model.

Much research on classical  
bin packing concerns the one-dimensional case, 
in which both item and bin sizes are scalars. 
It is possible to generalize one-dimensional bin packing to higher dimensions 
in multiple ways. For example, in vector packing problems \cite{CK2004, St2012}, 
item and bin sizes are vectors, and in box packing problems \cite{Bansal2009}, 
items and bins are rectangles or hyper-rectangles. Let us also remark that the packing constraints in our model
include vector packing and box packing as special cases. 
See e.g., \cite{Bansal2009} for an overview of multi-dimensional packing.

Motivated partly by applications to computer storage allocation, 
a {\em one-dimensional, dynamic bin packing} problem was introduced in \cite{CGJ1983}, 
where, similar to the model of this paper, 
items leave the system after a finite service time. 
Paper \cite{CGJ1983} contains a worst-case analysis of the problem, 
so the techniques used are quite distinct from those in our paper. 
A recent review of results on worst-case dynamic bin packing can be found in \cite{CCGMV2013}. 

Another related line of works considers bin packing {\em service} systems, which have one (see e.g. \cite{CS2001, Gamarnik2004}) 
or several servers (see e.g., \cite{Jiang2012, Maguluri2012, Maguluri2013, GSW2012}). 
In these systems, random-size items (or customers) arrive over time, and get placed into 
servers for processing, subject to packing constraints. 
Customers waiting for service are queued, and a typical problem is 
to determine the maximum throughput and/or minimum queueing delay 
under a packing algorithm. 
Our model is similar to these systems, since they model customer departures, 
but is also different, mainly because our system has an infinite number of servers, 
so that there are no queues or problem of stability.
Like our work, 
recent papers on bin packing service systems 
with multiple servers \cite{Jiang2012, Maguluri2012, Maguluri2013, GSW2012}
are also motivated by real-time VM placement problems.

As mentioned in the introduction, 
the model in this paper is the same as that studied in 
\cite{St2012, StZh2012, StZh2013}.
Papers 
\cite{St2012,StZh2012} introduce and study
different classes of {\em Greedy} algorithms,
and prove their asymptotic optimality, 
as the system scale grows to infinity.
A Greedy algorithm does not 
use the knowledge of the customer arrival rates
or mean service times, and makes placement decisions based on the current
system state only. However, unlike GRAND, it does need to keep track of the numbers
of servers 
in different {\em packing configurations}. 
The number of possible configurations can
be prohibitively 
large in many practical scenarios, a feature that may limit the implementability of Greedy.

The relation of our main results to those for GRAND($aZ$) in \cite{StZh2013} has already been discussed in much detail. 
Also related to the GRAND algorithms 
is a recent paper \cite{St2015_grand-het}, 
which generalizes 
GRAND($aZ$) 
and its asymptotic optimality 
to {\em heterogeneous} systems with multiple server types, in which packing constraints depend on the server type;
it also contains results for heterogeneous systems with {\em finite pools} of servers of each type.
Finally, a randomized version of the Best-Fit algorithm was studied in \cite{GZS2014}, which considers 
the model of this paper with specialized packing constraints, and 
was proved to be asymptotically optimal, using techniques similar to those in \cite{StZh2013}.

\subsection{Organization}

The rest of the paper is organized as follows.
In 
Subsection~\ref{subsec-notation}, we introduce basic notation and conventions that will be used throughout the paper.
The model, the GRAND($Z^p$) algorithm and the asymptotic regime are formally defined in Section~\ref{sec-model}.
We  state Theorem~\ref{thm-zp}, the main result on the
asymptotic optimality of GRAND($Z^p$), in Section~\ref{sec:main-thm}. 
In Section~\ref{sec-prelim}, 
we collect some results and observations obtained in \cite{StZh2013},
which are needed for our proofs.
The proof of  Theorem~\ref{thm-zp} is in Section~\ref{sec-grand-zp}, with the proofs of some auxilliary results given in the Appendix.
The paper is concluded in Section~\ref{sec-discuss} with some discussion and suggestions for future work.

\subsection{Basic Notation and Conventions}
\label{subsec-notation}

Sets of real and non-negative real numbers are denoted by $\R$ and $\R_+$, respectively. 
Similarly, sets of integers and non-negative integers are denoted by $\Z$ and $\Z_+$, respectively. 
For $\xi \in \R$, $\lceil \xi \rceil$ denotes the smallest integer
{greater than or equal to} $\xi$, and $\lfloor \xi \rfloor$ denotes the largest integer 
smaller than or equal to $\xi$.
For two vectors $\bx, \by \in \R^n$, we use $\bx \cdot \by$ to 
denote their scalar (dot) product; i.e., $\bx \cdot \by = \sum_i x_i y_i$.
The standard Euclidean norm of a vector $\bx\in \R^n$ is denoted by $\|\bx\| = \sqrt{\bx \cdot \bx}$, 
and the distance from vector $\bx$ to a set $U$ in a Euclidean space is denoted by 
$d(\bx,U)=\inf_{\bu\in U} \|\bx-\bu\|$. 
$\bx \to \bu \in \R^n$ means ordinary convergence in $\R^n$,
and $\bx \to U \subseteq \R^n$
means $d(\bx,U) \to 0$. 
For $\bx \in \R^n$, we also use $\|\bx\|_1$ to denote the $1$-norm of $\bx$, defined to be $\|\bx\|_1 = \sum_i |x_i|$. 
$\be_i$ is the $i$-th coordinate unit vector in $\R^n$. 
Symbol $\implies$
denotes convergence in distribution of random variables taking values in space $\R^n$
equipped with the Borel $\sigma$-algebra. 
Symbol $\stackrel{d}=$ means {\em equal in distribution}. 
The abbreviation {\em w.p.1} means  
{\em with probability 1}. 
The abbreviation RHS (LHS, respectively) means {\em right-hand side} ({\em left-hand side}, respectively). 
In addition, the abbreviation {\em w.r.t} means {\em with respect to}, 
and abbreviation {\em u.o.c.} means 
{\em uniform on compact sets}. 
We often write $x(\cdot)$ to mean the function (or random process) $\{x(t),~t\ge 0\}$,
and we write $\{x_{\bk}\}$ to mean the vector $\{x_{\bk}, ~\bk\in\ck\}$, 
where the set of indices $\ck$ is determined by the context.
For a function (or random process) $x(\cdot)$, 
we use $x(t+)$ to denote its right limit at time $t$, i.e., $x(t+) = \lim_{s \downarrow t} x(s)$, 
and use $x(t-)$ to denote its left limit at time $t$, i.e., $x(t-) = \lim_{s \uparrow t} x(s)$, whenever these limits exist. 
For a differentiable function $f :\R^n \to \R$, its gradient at $\bx \in \R^n$ is denoted by
$\nabla f(\bx) = \left\{\frac{\partial}{\partial x_i} f(\bx), ~i=1,\ldots,n\right\}$.
Indicator function $I\{A\}$ for a condition $A$ is equal to $1$
if $A$ holds and $0$ otherwise. 
The cardinality of a finite set $\mathcal{N}$ is 
$|\mathcal{N}|$. 
Notation $\doteq$ means {\em is defined to be}. In this paper, we use bold letters to represent vectors, 
and plain letters to represent scalars.

\section{System Model}
\label{sec-model}

\subsection{Infinite Server System with Packing Constraints}

The model that we study in this paper is the same as in \cite{St2012, StZh2012, StZh2013}. 
We consider a service system with an infinite number of servers. 
There are $I$ types of arriving customers, indexed by $i \in \{1,2,\ldots,I\} \doteq \ci$. 
For each $i \in \ci$, customers of type $i$ arrive as an independent Poisson process of rate $\Lambda_i >0$, 
and have service times that are exponentially distributed with mean $1/\mu_i$. 
The service times of all customers are mutually independent, and independent of the arrival processes. 
Each customer is placed into one of the servers for processing, immediately upon arrival, 
and departs the system after the service completes. 
Multiple customers can occupy the same server simultaneously, 
subject to the so-called {\em monotone} packing constraint, which we now describe. 
\begin{definition}[Monotone packing]\label{def:monotone}
A packing constraint is characterized by a finite set $\bar\ck$, the set of feasible server {\em configurations}. 
A vector $\bk = \{k_i, ~i\in \ci\} \in \bar\ck$ is a (feasible) server configuration 
if (a) for each $i \in \ci$, $k_i \in \Z_+$; and (b) 
a server can simultaneously process $k_1$ customers of type $1$, $k_2$ customers of type $2$, $\cdots$, 
and $k_I$ customers of type $I$. 
The packing constraint $\bar\ck$ is called {\em monotone}, if the following condition holds: 
whenever $\bk \in \bar\ck$ and $\bk' \leq \bk$ componentwise, then $\bk' \in \bar\ck$ as well. 
\end{definition}
Without loss of generality, we assume that $\be_i \in \bar\ck$ for all $i \in \ci$, where $\be_i$ is the $i$-th coordinate unit vector, 
so that customers of all types can be processed.
By definition, the component-wise zero vector $\bZero$ belongs to $\bar\ck$ -- this is the configuration of an empty server. We denote by $\ck=\bar\ck \setminus \{\bZero\}$ 
the set of configurations that {\em do not} include the zero configuration. 
As discussed in \cite{St2012, StZh2012}, monotone packing includes important special packing constraints such as vector packing. 

An important assumption of the model is that simultaneous service does {\em not} affect the service time distributions 
of individual customers. 
In other words, the service time of a customer is unaffected by whether or not there are other customers served simultaneously by the same server.

We now define the system state. For each $\bk \in \ck$, let $X_{\bk}(t)$ denote 
the number of servers with configuration $\bk$. 
Then, the vector $\bX(t) = \{X_{\bk}(t), \bk \in \ck\}$ is the {\em system state} at time $t$, 
and we often write $\bX = \{X_{\bk}, \bk \in \ck\}$ for a generic state. 
Note that the system state does not include the number of empty servers, which would always be infinite.

The following notation and terminology will be used in the sequel. 
Define $\cm = \{(\bk, i) : \bk \in \ck, \bk - \be_i \in \bar\ck\}$. 
Note that $(\bk, i) \in \cm$ are in one-to-one correspondence with the pairs $(\bk, \bk-\be_i)$ 
with $\bk \in \ck$ and $\bk - \be_i \in \bar\ck$, 
so $(\bk, i)$ can be viewed as a shorthand for $(\bk, \bk-\be_i)$. 
For this reason, we call a pair $(\bk, i)\in \cm$ an {\em edge}.
The placement of a type-$i$ arrival into a server with configuration 
$\bk-\be_i$ to form configuration $\bk$ is called 
{\em an arrival along the edge $(\bk, i)$}, 
and the departure of a type-$i$ customer 
from a server with configuration $\bk$, which changes the server configuration to $\bk-\be_i$, 
is called {\em a departure along the edge $(\bk, i)$}.

\subsection{GRAND and GRAND($Z^p$) Algorithms}
\label{ssec-grand-algo}

In general, a {\em placement (or packing) algorithm} 
determines the servers into which arriving customers are placed dynamically over time. 
In this paper, we are 
interested in {\em online} placement algorithms, 
which make placement decisions based only on the current system state $\bX$. 
Thus, from now on, we will use ``placement algorithms'' to mean online ones. 
Our primary objective is the design of placement algorithms that minimize
the total number of  occupied servers  $\sum_{\bk\in\ck} X_{\bk}$ 
in the stationary regime. 

Under any well-defined placement algorithm, 
the process $\{\bX(t), t\ge 0\}$ is a continuous-time Markov chain 
with a countable state space, which is irreducible and positive recurrent. 
The positive recurrence of $\{\bX(t), t\ge 0\}$ can be argued as follows. 
First, for each $i \in \ci$ and each $t \geq 0$, we let 
\begin{equation}\label{eq:type-i-no}
Y_i(t) = \sum_{\bk \in \ck} k_i X_{\bk}(t)
\end{equation}
be the number of type-$i$ customers in the system at time $t$, 
and let 
\begin{equation}\label{eq:total-no}
Z(t) = \sum_i Y_i(t)
\end{equation}
be the total number of customers in the system at time $t$. 
For each $i \in \ci$, $\{Y_i(t): t\geq 0\}$ describes exactly the dynamics 
of an independent $M/M/\infty$ queueing system with arrival rate $\Lambda_i$ and service rate $\mu_i$, 
regardless of the placement algorithm, and has a unique stationary distribution. 
Let us denote by $Y_i(\infty)$ the random value of $Y_i(t)$ in the stationary regime. 
Then, $Y_i(\infty)$ is a Poisson random variable with mean $\Lambda_i/\mu_i$ for each $i \in \ci$.  
Similarly, the process $\{Z(t): t\geq 0\}$ also has a unique stationary distribution, 
and if we let $Z(\infty)$ be the random value of $Z(t)$ in the stationary regime, then $Z(\infty)$ is 
a Poisson random variable with mean $\sum_{i \in \ci} \Lambda_i/\mu_i$. 
Thus, the positive recurrence of the 
Markov chain $\{\bX(t), t\ge 0\}$ 
follows from the facts that 
$\sum_{\bk \in \ck} X_{\bk}(t) \leq \sum_i Y_i(t) = Z(t)$ 
for all $t \geq 0$, and that the process $\{Z(t): t\geq 0\}$ is positive recurrent. 
Consequently, the process $\{\bX(t), ~t\ge 0\}$ has a unique stationary distribution, 
and we let $\bX(\infty) = \{X_{\bk}(\infty), \bk \in \ck \}$ be the random system state $\bX(t)$ 
in the stationary regime.

In \cite{StZh2013}, we introduced a broad class of placement algorithms, 
called the Greedy-Random (GRAND) algorithms. 
For completeness, we include the definition here. 

\begin{definition}[Greedy-Random (GRAND) algorithm]\label{df:grand}
At any given time $t$, there is a designated finite set of $X_{\bZero}(t)$ empty servers, 
called {\em zero-servers}, where 
$X_{\bZero}(t) = f(\bX(t))$ is a given fixed function 
of the system state $\bX(t)$. 
Suppose that a type-$i$ customer arrives at time $t$. 
Then, the customer is placed into a server 
chosen uniformly at random among the zero-servers and the occupied servers 
where the customer can fit. In other words,
the total number of servers available to a type-$i$ arrival at time $t$ is
$$
X_{(i)}(t) \doteq X_{\bZero}(t) + \sum_{\bk\in \ck:~\bk+\be_i\in \ck} X_{\bk}(t).
$$
If $X_{(i)}(t)=0$, then the customer is placed into an empty server. 
Furthermore, it is important
that immediately after any arrival or departure at time $t$, 
the value of $X_{\bZero}(t+)$ 
is reset (if necessary) to  $f(\bX(t+))$.
\end{definition}

Note that the number $X_{\bZero} = f(\bX)$ of zero-servers is finite at all times, even though
there is always an infinite number of empty servers available.
Also, recall that 
a system state $\bX = \{X_{\bk}, ~\bk\in \ck\}$ only 
includes the quantities of servers in non-zero configurations, so it does not include $X_{\bZero}$.
Clearly, for a given function $f(\cdot)$, GRAND is a well-defined placement algorithm.

The focus of this paper is the following specialization of the general GRAND algorithm. 

\begin{definition}[GRAND($Z^p$) algorithm]\label{df:grand-az}
A special case of the GRAND algorithm, with 
the number of zero-servers depending only on the 
total number of customers as
$$
X_{\bZero}(t) =\lceil  Z(t)^p \rceil  \quad \forall t \geq 0,
$$ 
where $p\in (0, 1)$ is a parameter, is called a GRAND($Z^p$) algorithm.
\end{definition}

In \cite{StZh2013}, we introduced a different specialization of the general GRAND algorithm, 
namely the GRAND($aZ$) algorithm, where the number of zero-servers $X_{\bZero}$ 
depends on $Z$, the total number of customers, as $X_{\bZero} = \lceil aZ \rceil$, $a > 0$. 
GRAND($Z^p$) is similar to GRAND($aZ$) in that under both 
algorithms, 
$X_{\bZero}$ only depends on the system state $\bX$ through $Z$, the total number of customers. 

As we can see, GRAND are fairly ``blind'' algorithms in that
when they place a customer, they do 
not prefer one configuration over another, as long as they can fit this additional customer. 
Similar to GRAND($aZ$), GRAND($Z^p$) can be efficiently implemented. 
At all times, the algorithm only needs to keep track of $I+1$ variables to make placement decisions, 
with the variables being (a) $Z$, the total number of customers; and (b) 
for each $i \in \ci$, $X_{(i)}$, 
the total number of servers that can fit an additional type-$i$ customer, 
including zero-servers and occupied ones. 
In contrast to GRAND($Z^p$) and GRAND($aZ$), 
the Greedy algorithm \cite{St2012} and the Greedy-with-Sublinear-safety-Stocks (GSS) 
algorithm \cite{StZh2012} both need 
to keep track of $|\ck|$ number of variables, 
which is often  
prohibitively large in real applications.

\subsection{Asymptotic Regime}

We are interested in the asymptotic properties 
of the GRAND($Z^p$) algorithms, as the arrival rates become large.
Specifically, assume that a positive {\em scaling parameter} $r$ increases to infinity along a discrete
sequence. 
Customer arrival rates scale linearly with $r$; i.e., 
for each $r$, $\Lambda_i = \lambda_i r$, where $\lambda_i$ are fixed positive parameters. 
Without loss of generality, we assume that $\sum_i \lambda_i/\mu_i = 1$, 
by suitably re-defining $r$ if required. 
For each $r$, let $\bX^r(\cdot)$ be the random process associated with the system parametrized by $r$, 
{let} $\bX^r(\infty)$ be the (random) system state in the stationary regime, 
and let $X_{\bZero}^r(t)$ be the number of zero-servers in the system at time $t$.
For each $i \in \ci$, let $Y^r_i(t) = \sum_{\bk\in\ck} k_i X^r_{\bk}(t)$ be the total number
of type-$i$ customers at time $t$, and let $Y^r_i(\infty) = \sum_{\bk\in\ck} k_i X^r_{\bk}(\infty)$. 
Similarly, let $Z^r(t) = \sum_{i \in \ci} Y^r_i(t)$ 
be the total number of customers at time $t$ and let $Z^r(\infty) = \sum_i Y^r_i(\infty)$. 
As explained in Section \ref{ssec-grand-algo}, 
for each $i \in \ci$, 
$Y^r_i(\infty)$ is an independent Poisson random variable with mean $r \rho_i$, where $\rho_i\equiv \lambda_i/\mu_i$, 
and $Z^r(\infty)$ is a Poisson random variable with mean $r \sum_i \rho_i = r$.

The {\em fluid-scaled} processes are defined to be $\{\bx^r(t) : t\geq 0\}$ for each $r$, 
where $\bx^r(t) = \bX^r(t)/r$. 
We also define $\bx^r(\infty) = \bX^r(\infty)/r$.
For any $r$, $\bx^r(t)$ takes values in the  non-negative orthant $\R_+^{|\ck|}$.
Similarly, $y^r_i(t)\doteq Y^r_i(t)/r$, $z^r(t) \doteq Z^r(t)/r$, $x^r_{\bZero}(t) \doteq X^r_{\bZero}(t)/r$ and 
$x^r_{(i)}(t) \doteq X^r_{(i)}(t)/r$,  for  $t \in [0, \infty]$.
Since $\sum_{\bk\in \ck}  x_{\bk}^r(\infty) \leq \sum_i y^r_i(\infty) \leq z^r(\infty)=Z^r(\infty)/r$,
the 
family of random variables
$\sum_{\bk\in \ck} x_{\bk}^r(\infty)$ is uniformly integrable in $r$.
This in particular implies that the sequence of distributions of $\bx^r(\infty)$ is tight,
so there always exists a limit {$\bx(\infty)$} in distribution{, and} $\bx^r(\infty)\implies \bx(\infty)$,
along a subsequence of $r$. 

Since $Y^r_i(\infty)$ is a Poisson random variable of mean $\rho_i r$, 
any weak limit point $y_i(\infty)$ of $\left\{y^r_i(\infty)\right\}_r$ 
must concentrate at the constant $\rho_i$.
Thus, the limit (random) vector $\bx(\infty)$ satisfies the following conservation laws:
\beql{eq-cons-laws}
\sum_{\bk\in\ck} k_i x_{\bk}(\infty) \equiv y_i(\infty) = \rho_i, ~~\forall i,
\end{equation}
{which,} in particular, {implies that}
\beql{eq-cons-laws2}
z(\infty)\equiv \sum_i y_i(\infty) \equiv \sum_i \rho_i
= 1.
\end{equation}
Therefore, the values of $\bx(\infty)$ are confined to the convex compact 
$(|\ck|-I)$-dimensional 
polyhedron
\begin{equation}\label{eq-feasible-set}
\cx \doteq \left\{\bx\in \R_+^{|\ck|} ~\Big|~  \sum_{\bk} k_i x_{\bk} = \rho_i, ~\forall i\in\ci \right\}.
\end{equation}
We will slightly abuse notation by using symbol $\bx$ for a generic element of $\R_+^{|\ck|}$;
while $\bx^r(\infty), \bx(\infty),\bx^r(t)$, and later $\bx(t)$, refer to random  vectors taking values in $\R_+^{|\ck|}$.

Note that the asymptotic regime and the associated basic properties \eqn{eq-cons-laws}
and \eqn{eq-cons-laws2} hold {\em for any  {placement algorithm}}. 
 Indeed, \eqn{eq-cons-laws}
and \eqn{eq-cons-laws2} only depend on the fact 
that $Y_i^r(\infty)$ are mutually independent
Poisson random variables
with means $\rho_i r$, discussed earlier.

Consider the following linear program (LP) of minimizing $\sum_{\bk\in\ck} x_{\bk}$, 
the number of occupied servers on the fluid scale.
\begin{flalign}
 \text{(LP)} & &  \text{Minimize } & \sum_{\bk\in\ck} x_{\bk} & \label{eq-opt} \\
&&\mbox{subject to } & \sum_{\bk\in\ck} k_i x_{\bk} = \rho_i ~~\forall i \in \ci, &\label{eq-cons-laws222} \\
&& & \quad \quad \quad x_{\bk} \geq 0, ~~\forall \bk \in \ck. \label{eq-non-neg}
\end{flalign} 
Denote by $\cx^* \subseteq \cx$ the set of its optimal solutions, 
and by $L^*$ its optimal value.
Since constraints \eqref{eq-cons-laws222} hold for $\bx(\infty)$ 
under any placement algorithm, 
the optimal value $L^*$ provides a lower bound 
on $\sum_{\bk} x_{\bk}(\infty)$, 
the steady-state total number of occupied servers on the fluid scale, 
under any placement algorithm.

\section{Main Result}
\label{sec:main-thm}

The {\bf main result} of this paper is the following theorem. It shows that GRAND($Z^p$)
is asymptotically optimal, when parameter $p$ is sufficiently close to $1$, 
which depends on the structure of the packing constraint specified by $\ck$.

\begin{thm}
\label{thm-zp}
Denote $\kappa \doteq 1+\max_{\bk} \sum_i k_i$, and assume that 
$p<1$ and 
\beql{eq-cond-p}
1-\kappa (1-p) > 7/8, ~~~\mbox{or equivalently,}~~~ p > 1-1/(8\kappa).
\end{equation}
Consider a sequence of systems, with parameter $r\to\infty$, operating 
under the GRAND($Z^p$) algorithm.
For each $r$, let $\bx^r(\infty)$ denote the random state 
of the fluid-scaled process in the stationary regime. 
Then as $r \rightarrow \infty$, 
\[
d(\bx^r(\infty), \cx^*) \Rightarrow 0.
\]
\end{thm}

We would like to compare our main result, Theorem \ref{thm-zp}, 
with the main results (Theorems 3 and 4) of \cite{StZh2013}, 
on the asymptotic performance of the GRAND($aZ$) algorithm. 
For any given $p$ that satisfies \eqref{eq-cond-p}, 
GRAND($Z^p$) is asymptotically optimal 
in the sense that the {\em fluid-scaled optimality gap}, 
as measured by $d(\bx^r(\infty), \cx^*)$, goes to $0$ 
as the system scale $r$ grows to infinity. 
In contrast, GRAND($aZ$) is asymptotically optimal in the following (essentially, weaker)
sense: for any fixed parameter $a>0$ and under GRAND($aZ$), as $r \to \infty$,
the fluid-scaled optimality gap $d(\bx^r(\infty), \cx^*)$ converges to a constant $c(a)$,
which is not necessarily (and typically is not) zero; however, $c(a) \to 0$ as $a \to 0$. 
So, speaking informally, GRAND($Z^p$) optimality requires only that $r\to\infty$, while the optimality of
GRAND($aZ$) requires that $r\to\infty$ and then $a\to 0$.

Let us also remark on the role of  condition
\eqn{eq-cond-p} on the parameter $p \in (0, 1)$ in Theorem~\ref{thm-zp}.
This condition is technical, needed for our technical  approach to work.
It requires $p$ to be {\em sufficiently close} to $1$. Roughly speaking, 
when the parameter $p$ is sufficiently close to $1$, this guarantees that 
$x_{\bZero}^r$ is sufficiently ``large,'' approximately $r^{p-1}$. 
This induces the steady-state values of 
$x_{\bk}^r$ to be also ``large'', at least $r^{7/8 -1}$, for {\em all} configurations $\bk$
(see Lemma \ref{lem-p2}).
This in turn allows us to obtain desired bounds on the drift and increments of 
an appropriate Lyapunov function under GRAND($Z^p$) 
(see Sections \ref{sec-lyap-drift} and \ref{sec-complete-thm-zp} for details).

\section{Preliminaries} 
\label{sec-prelim}

In this section, we present some definitions and facts from \cite{StZh2013}, 
to provide enough background  
needed for the proof of Theorem \ref{thm-zp}, our main result. 
 The main purpose of this section is to introduce 
the Lyapunov function $L^{(a)}$ defined in \eqn{eq-L-def}, parameterized by $a>0$, 
together with its ``drift'' term $\Xi$, defined in \eqn{eq-L-deriv}, 
when the system is operating under the corresponding GRAND algorithm. 
In \cite{StZh2013}, the term $\Xi$ was shown to be the time derivative of $L^{(a)}$ 
in the fluid limit, under the GRAND($aZ$) algorithm. For the GRAND($Z^p$) algorithm that
we study in this paper, we will use the Lyapunov function $L^{(a)}$ with $a$ depending on $r$, namely
$a = r^{p-1}$, to prove Theorem \ref{thm-zp};
in this case the term $\Xi$ approximates the drift of $L^{(a)}$, and for any large fixed $r$ the random increment of 
$L^{(a)}$ can be shown to be well approximated by a time integral of $\Xi$, with high probability.
Along the way, 
we also introduce some useful properties of the LP \eqn{eq-opt}--\eqn{eq-non-neg}, $L^{(a)}$ and $\Xi$.

\paragraph{Dual characterization of (LP).}
Using the monotonicity of $\bar{\ck}$ (cf. Definition \ref{def:monotone}), it is easy to check that if, in the program (LP) defined by  
\eqn{eq-opt}-\eqn{eq-non-neg}, we replace equality constraints
\eqn{eq-cons-laws222} with the inequality constraints
\beql{eq-cons-laws222ge}
\sum_{\bk\in\ck} k_i x_{\bk} \ge \rho_i, ~~\forall i,
\end{equation}
we form a new linear program (LP') with the same optimal value as (LP). 
More explicitly, (LP') is given by 
\begin{flalign*}
 \text{(LP')} & &  \text{Minimize } & \sum_{\bk\in\ck} x_{\bk} & \\
&&\mbox{subject to } & \sum_{\bk\in\ck} k_i x_{\bk} \geq \rho_i ~~\forall i \in \ci, \\
&& & \quad \quad \quad x_{\bk} \geq 0 ~~\forall \bk \in \ck. 
\end{flalign*} 
Let $\cx^{**}$ denote the set of optimal solutions of (LP'). Then, 
$\cx^{**}$ contains $\cx^*$, the set of optimal solutions of (LP); 
or more precisely, $\cx^* = \cx^{**} \cap \cx$.
The dual program (DUAL') of (LP') is given by 
\begin{flalign}
 \text{(DUAL')} & &  \text{Maximize } & \sum_{i \in \ci} \rho_i \eta_i & \\
&&\mbox{subject to } & \sum_{i \in \ci} k_i \eta_i \leq 1, ~~\forall \bk \in \ck, \label{eq-dual-1}\\
&& & \quad \quad \quad \eta_i \geq 0, ~~\forall i \in \ci. \label{eq-dual-111}
\end{flalign} 
The following lemma is a simple consequence of 
the Kuhn-Tucker theorem.
\begin{lem}\label{lem-dual'} 
Vector $\bx$ is an optimal solution of (LP), i.e., $\bx \in \cx^*$, 
if and only if $\bx\in \cx$,
and there exists a vector $\bbeta=\{\eta_i, ~i\in \ci\}$ that satisfies \eqref{eq-dual-1}, \eqref{eq-dual-111} 
and the complementary slackness condition: 
\beql{eq-dual-3}
\mbox{for any $\bk \in \ck$, if}
~\sum_i k_i \eta_i < 1,~\mbox{then}~x_{\bk} = 0.
\end{equation}
(Clearly, any such vector $\bbeta$ is an optimal solution of (DUAL').)
\end{lem}
Let $\ch^*$ be the set of vectors $\bbeta$ that satisfy \eqn{eq-dual-1}-\eqn{eq-dual-3} for some $\bx \in \cx$. 

\paragraph{A useful Lyapunov function.} 
For any parameter $a \in (0, 1)$, define the function $L^{(a)}: \R_+^{|\ck|} \to \R$ 
by 
\beql{eq-L-def}
L^{(a)}(\bx) \doteq -\frac{1}{\log a}\sum_{\bk\in\ck} x_{\bk} \log \left(\frac{x_{\bk} c_{\bk}}{e a}\right),
\end{equation}
where $c_{\bk} \doteq \prod_i k_i !$, $0!=1$, and we use the convention that $0\log 0 = 0$. 

Function $L^{(a)}(\bx)$ 
can be viewed as as an approximation 
to the linear function $\sum_{\bk \in \ck} x_{\bk}$. 
Indeed, it has been shown \cite{StZh2013} 
that as $a \to 0$, $L^{(a)}(\bx) \to \sum_{\bk \in \ck} x_{\bk}$ 
uniformly over any compact set. 

Function $L^{(a)}(\bx)$  was introduced in \cite{StZh2013}  in the analysis 
of the GRAND($aZ$) algorithm. Its unique optimal point $\bx^{*,a}$ has a product form (see Lemma \ref{lem-product-form-cvx} below),
and $L^{(a)}(\bx)$ served as a Lyapunov function to show that in the fluid limit, $\bx^{*,a}$ is the unique equilibrium point under GRAND($aZ$). 
The rest of this section is devoted to stating properties of -- or related to -- 
function $L^{(a)}(\bx)$. For more details on the intuition for the Lyapunov function $L^{(a)}(\cdot)$, see Remark 1 of \cite{StZh2013}.
Note that in the analysis of GRAND($Z^p$) in this paper, $L^{(a)}(\bx)$ will also serve as a Lyapunov function, but with the parameter
 $a$ now depending on $r$ as $a = r^{p-1}$ -- this is the source of one of the key challenges we face in this paper.

Throughout the paper, we will use notation $b= -\log a$. 
Then, for each $\bk \in \ck$, we have
\beql{eq-L-partial}
\frac{\partial}{\partial x_{\bk}} L^{(a)}(\bx) = \frac{1}{b}  \log \left(\frac{c_{\bk}x_{\bk}}{a}\right).
\end{equation}
Note that if we adopt the convention that
\beql{eq-formal-deriv-zero}
\frac{\partial}{\partial x_{\bZero}} L^{(a)}(\bx)\Big|_{x_{\bZero}=a} = 0,
\end{equation}
then \eqn{eq-L-partial} is valid for $\bk=\bZero$ and $x_{\bZero}=a$, which will be useful later.

Function $L^{(a)}$ 
is strictly convex in $\bx\in\R_+^{|\ck|}$.
Consider the problem $\min_{\bx\in \cx} L^{(a)}(\bx)$. It is the
 following convex optimization problem (CVX($a$)):
\begin{flalign}
 \text{(CVX($a$))} & &  \text{Minimize } & L^{(a)}(\bx) & \label{eq-opt-grand} \\
&&\mbox{subject to } & \sum_{\bk\in\ck} k_i x_{\bk} = \rho_i, ~~\forall i \in \ci, &\label{eq-cons-laws-grand} \\
&& & \quad \quad \quad x_{\bk} \geq 0, ~~\forall \bk \in \ck. \label{eq-non-neg-grand}
\end{flalign}
Denote by $\bx^{*,a} \in \cx$ its unique optimal solution. 
The following lemma 
provides a crisp characterization of the point $\bx^{*, a}$. 
\begin{lem}\label{lem-product-form-cvx}
A point $\bx\in \cx$ is the optimal solution to (CVX($a$)) defined by \eqn{eq-opt-grand}-\eqn{eq-non-neg-grand}, 
i.e., $\bx=\bx^{*,a}$, if and only if it has a product form representation
\beql{eq-grand-product}
x^{*,a}_{\bk} = \frac{a}{c_{\bk}} \exp \left[b \sum_i k_i \nu_i^{*,a}\right] = \frac{1}{c_{\bk}} a^{1-\sum_i k_i \nu_i^{*,a}} , ~~ \forall \bk \in \ck,
\end{equation}
for some vector $\bnu^{*,a} = \{\nu_i^{*,a}, ~ i\in \ci\}$.
\end{lem}
\begin{proof} 
For each $i \in \ci$, let $\nu_i$ be the dual variable that corresponds to the equality constraint $\sum_{\bk\in \ck} k_ix_{\bk} = \rho_i$, 
and consider the Lagrangian
$$
L^{(a)}(\bx) + \sum_i \nu_i \left(\rho_i - \sum_{\bk\in\ck} k_i x_{\bk}\right).
$$
By setting partial derivatives of the Lagrangian 
to zero, we get 
\[
\frac{1}{b}  \log \left[\frac{x_{\bk} c_{\bk}}{a} \right] - \sum_i \nu_i k_i = 0,  ~~\forall \bk \in \ck.
\]
$\bx^{*,a} \in \cx$ is the optimal solution to (CVX($a$)) if and only if there exist Lagrange multipliers 
$\bnu^{*,a} = \{\nu_i^{*,a}, ~ i\in \ci\}$ 
such that 
\[
\frac{1}{b}  \log \left[\frac{x_{\bk}^{*,a} c_{\bk}}{a} \right] - \sum_i \nu_i^{*,a} k_i = 0,  ~~\forall \bk \in \ck, 
\]
which leads to the product form representation \eqn{eq-grand-product}.
\end{proof}

A simple consequence of Lemma \ref{lem-product-form-cvx} is that the Lagrange multipliers $\nu_i^{*,a}$ are unique 
and are equal to $1 - \log \left(x_{\be_i}^{*,a}\right)/\log a$, by considering \eqref{eq-grand-product} for $\be_i$, $i \in \ci$.
The following result is from \cite{StZh2013}.
\begin{thm}
\label{th-grand-fluid-convergence}
Let $\bx^{*,a}$ and $\bnu^{*,a}$ be as in Lemma \ref{lem-product-form-cvx}.
Then, as $a\downarrow 0$, $\bx^{*,a} \to \cx^*$ and $\bnu^{*,a} \to \ch^*$.
\end{thm}
Note that the convergence $\bx^{*,a} \to \cx^*$ is an easy consequence of the fact that 
$L^{(a)}(\bx)-\sum_{\bk \in \ck} x_{\bk}\to 0$ u.o.c., as $a\to 0$.

In \cite{StZh2013}, we considered {\em fluid limit} trajectories $\bx(\cdot)$, which, speaking informally,
arise as limits of the fluid-scaled process $\bx^r(\cdot)$ as $r\to\infty$.
We showed there that for any fixed $a > 0$, under GRAND($aZ$), 
the derivative $(d/dt)L^{(a)}(\bx(t))$ along a fluid limit trajectory 
$\bx(\cdot)$ is always negative, whenever $\bx(t) \neq \bx^{*,a}$;
consequently, $L^{(a)}$ served as a Lyapunov function, which allowed us to
prove that $\bx(t)\to \bx^{*,a}$, as $t\to\infty$, along any fluid limit trajectory.
This in turn was the key property that led to establishing the fact that the random steady-state system state $\bx^r(\infty)$ concentrates on $\bx^{*,a}$ as $r\to\infty$; together with Theorem \ref{th-grand-fluid-convergence} this means the asymptotic optimality of GRAND($aZ$), if the limit $r\to\infty$ of steady states $\bx^r(\infty)$ is taken first,
and the limit $a\to 0$ is taken after that. In this paper
we will also make use of the family 
of Lyapunov functions $L^{(a)}(\bx)$ defined in \eqn{eq-L-def}, 
but use them in a different way, because the analysis of GRAND($Z^p$) cannot (as will be explained later)
rely on an analysis of fluid-limit trajectories.

We continue to introduce 
some expressions and their properties, derived in \cite{StZh2013} 
and which are closely related to the Lyapunov function $L^{(a)}$. In \cite{StZh2013}, 
the expression $\Xi(\bx)$ defined below 
in \eqn{eq-L-deriv} 
is 
the derivative of the Lyapunov function $L^{(a)}(\bx(t))$ with respect to time, along a fluid limit trajectory 
$\bx(\cdot)$. In the context of this paper, 
$\Xi$ will be interpreted and used differently; 
roughly speaking, it will be the value of the process generator when applied to $L^{(a)}$.

Consider $\bx\in \R_+^{|\ck|}$ such that $x_{\bk}>0, ~\forall \bk\in \ck$. 
For two edges $(\bk, i)$ and $(\bk', i)$, consider the ordered pair $\left((\bk, i), (\bk', i)\right)$, 
for which we define 
\begin{equation}
\label{eq:small-xi-orig}
\xi_{\bk,\bk',i} = \frac{1}{b}\big[\log (k'_i x_{\bk-\be_i} x_{\bk'}) - \log (k_i x_{\bk} x_{\bk'-\be_i}) \big]
\frac{k_i \mu_i x_{\bk} x_{\bk'-\be_i}}{x_{(i)}},
\end{equation}
where {\em by convention}, 
\beql{eq-convention-x0}
x_{\bZero} = a 
~~\mbox{and, correspondingly,}~~x_{(i)} \doteq x_{\bZero} + \sum_{\bk\in \ck:~\bk+\be_i\in \ck} x_{\bk},
\end{equation}
and 
$\xi_{\bk,\bk',i}$ is defined for
cases when either $\bk-\be_i=\bZero$ or $\bk'-\be_i=\bZero$ as well. 
It will be convenient,  for each ordered pair of edges $\left((\bk,i), (\bk',i)\right)$, 
to define 
\beql{eq-diff-log}
\chi_{\bk,\bk',i} (\bx) = \log (k'_i x_{\bk-\be_i} x_{\bk'}) - \log (k_i x_{\bk} x_{\bk'-\be_i}) =
\log \left[\frac{x_{\bk-\be_i}}{k_i x_{\bk}} \slash \frac{x_{\bk'-\be_i}}{k'_i x_{\bk'}}\right].
\end{equation}
Clearly, $|\chi_{\bk,\bk',i} (\bx)|$ is the log-scale distance between the ratios $x_{\bk-\be_i}/(k_i x_{\bk})$ and $x_{\bk'-\be_i}/(k'_i x_{\bk'})$;
in particular, $\chi_{\bk,\bk',i} (\bx) =0$ if and only if $x_{\bk-\be_i}/(k_i x_{\bk}) = x_{\bk'-\be_i}/(k'_i x_{\bk'})$. Note also that
\begin{equation}\label{eq:chi-p1}
\chi_{\bk,\bk',i} (\bx) = - \chi_{\bk',\bk,i} (\bx).
\end{equation}
Then, \eqn{eq:small-xi-orig} can be written as 
\begin{equation}\label{eq:small-xi}
\xi_{\bk,\bk',i} = \frac{k_i \mu_i x_{\bk} x_{\bk'-\be_i}}{b x_{(i)}} \chi_{\bk,\bk',i} (\bx). 
\end{equation}
and we have the following:
\begin{align}
\xi_{\bk,\bk',i} + \xi_{\bk',\bk,i} 
= &~\frac{k_i \mu_i x_{\bk} x_{\bk'-\be_i}}{b x_{(i)}} \chi_{\bk,\bk',i} (\bx) 
+ \frac{k_i \mu_i x_{\bk'} x_{\bk-\be_i}}{b x_{(i)}} \chi_{\bk',\bk,i} (\bx) \nonumber \\
= &~\frac{k_i \mu_i x_{\bk} x_{\bk'-\be_i}}{b x_{(i)}} \chi_{\bk,\bk',i} (\bx) 
- \frac{k_i \mu_i x_{\bk'} x_{\bk-\be_i}}{b x_{(i)}} \chi_{\bk,\bk',i} (\bx) \nonumber \\
= &~\frac{\mu_i}{bx_{(i)}} \chi_{\bk,\bk',i} (\bx)
[k_i x_{\bk} x_{\bk'-\be_i} - k'_i x_{\bk-\be_i} x_{\bk'}] \nonumber \\
= &~\frac{\mu_i}{bx_{(i)}} [\log (k'_i x_{\bk-\be_i} x_{\bk'}) - \log (k_i x_{\bk} x_{\bk'-\be_i}) ] 
[k_i x_{\bk} x_{\bk'-\be_i} - k'_i x_{\bk-\be_i} x_{\bk'}] 
\le 0. \label{eq-a35}
\end{align}
The inequality in \eqn{eq-a35} holds because the two terms in the square brackets are non-zero and have opposite signs, unless they both are $0$;
therefore, the inequality in \eqn{eq-a35}
is strict unless $k'_i x_{\bk-\be_i} x_{\bk'}=k_i x_{\bk} x_{\bk'-\be_i}$.

Finally, we define
\beql{eq-L-deriv}
\Xi(\bx) = \sum_i \sum_{\bk,\bk'} [\xi_{\bk,\bk',i} + \xi_{\bk',\bk,i}], 
\end{equation}
where the summation $\sum_{\bk, \bk'}$ is over all {\em unordered} pairs $\{\bk, \bk'\}$. 
Observe that
$\Xi(\bx)<0$ unless $\bx$ has a product form representation \eqn{eq-grand-product}, 
for some vector $\bnu^{*,a}$. 
This is because $\Xi(\bx) = 0$ if and only if the equality in \eqn{eq-a35} holds 
for all unordered pairs of edges $\{(\bk, i), (\bk', i)\}$, 
which is true if and only if $k'_i x_{\bk-\be_i} x_{\bk'}=k_i x_{\bk} x_{\bk'-\be_i}$ [i.e., $\chi_{\bk,\bk',i} (\bx)=0$] 
for all ordered pairs $((\bk, i), (\bk', i))$, 
a condition that is equivalent to a product form presentation \eqn{eq-grand-product} of $\bx$, 
for some $\bnu^{*,a}$. 
Let us also note that here the vector $\bnu^{*,a}$ need not correspond to 
the (unique) Lagrange multipliers of the program (CVX($a$)), 
because $\bx$ does not necessarily satisfy the equalities \eqn{eq-cons-laws-grand};
i.e., we do not necessarily have $\bx \in \cx$ (recall the definition of $\cx$ in \eqn{eq-feasible-set}). 
If $\bx \in \cx$, we will have $\bx=\bx^{*,a}$, exactly as in \eqn{eq-grand-product}.

We now provide a more concrete interpretation of \eqn{eq-L-deriv}. 
Consider a 
state $\bx$ and suppose that 
the time derivatives (or rates of changes) of
$\bx$ are 
defined 
by the following dynamics. 
Along each edge $(\bk,i)$, 
``mass'' is moving from
$x_{\bk}$ to $x_{\bk-\be_i}$ at the rate
\beql{eq-tilde-w}
\tilde w_{\bk,i} = k_i \mu_i x_{\bk}.
\end{equation}
(Value $\tilde{w}_{\bk,i}$ can be interpreted as the {\em fluid-scaled} rate 
of type-$i$ departures along the edge $(\bk,i)$.)
In addition, along each edge $(\bk,i)$, 
``mass'' is moving from
$x_{\bk-\be_i}$ to $x_{\bk}$ at the rate
\beql{eq-tilde-v}
\tilde v_{\bk,i} = \frac{\tilde \lambda_i x_{\bk-\be_i}}{x_{(i)}},
\end{equation}
where we denote 
\beql{eq-tilde-lambda}
\tilde \lambda_i = \sum_{\bk\in\ck} k_i \mu_i  x_{\bk},
\end{equation}
and $x_{(i)}$ is defined by \eqn{eq-convention-x0}. 
(Values $\tilde{v}_{\bk, i}$ and $\tilde{\lambda}_i$ defined 
in \eqn{eq-tilde-v} and \eqn{eq-tilde-lambda} can be interpreted as follows. 
If $\bx \in \cx$, then $\tilde \lambda_i = \lambda_i$ for all $i\in \ci$, by the equality constraints \eqn{eq-cons-laws-grand}, 
and $\tilde{v}_{\bk,i}$ of \eqn{eq-tilde-v} is 
the {\em fluid-scaled} rate of 
type-$i$ arrivals along the edge $(\bk,i)$, 
under the GRAND($Z^p$) algorithm. 
In general, this is not necessarily the case, 
since we consider a generic system state $\bx$, which need not belong to $\cx$.)

Equations \eqn{eq-tilde-w} and \eqn{eq-tilde-v} define the derivatives
of 
$x_{\bk}$, for all $\bk \in \ck$. 
By convention, $x_{\bZero}$ remains constant at $a$, so it has zero derivative. 
By \eqref{eq-L-partial}, 
let us note that in the expression of $\xi_{\bk,\bk',i}$ (recall \eqref{eq:small-xi-orig}), 
\[
\frac{1}{b}\big[\log (k'_i x_{\bk-\be_i} x_{\bk'}) - \log (k_i x_{\bk} x_{\bk'-\be_i}) \big] =
\]
\beql{eq-777}
\left[\frac{\partial}{\partial x_{\bk'}} L^{(a)}(\bx) - \frac{\partial}{\partial x_{\bk'-\be_i}} L^{(a)}(\bx)\right] 
- \left[\frac{\partial}{\partial x_{\bk}} L^{(a)}(\bx) - \frac{\partial}{\partial x_{\bk-\be_i}} L^{(a)}(\bx)\right].
\end{equation}
Then, by equations \eqref{eq-tilde-w} and \eqref{eq-tilde-v} (which define the derivatives of $x_{\bk}$), 
the convention \eqref{eq-convention-x0}, and equation \eqref{eq-777}, 
it is not difficult to check that $\Xi(\bx)$ is the time 
derivative of $L^{(a)}(\bx)$:
\beql{eq-xi-is-derivative}
\Xi(\bx) = \sum_{(\bk,i) \in \cm} \left[\frac{\partial}{\partial x_{\bk}} L^{(a)}(\bx) - \frac{\partial}{\partial x_{\bk-\be_i}} L^{(a)}(\bx)\right]
\big[\tilde v_{\bk,i} - \tilde w_{\bk,i}\big],
\end{equation}
where the convention \eqn{eq-formal-deriv-zero} is used.


\section{Proof of Theorem \ref{thm-zp}} 
\label{sec-grand-zp}

This section is organized as follows. 
In Subsection \ref{sec-lyap-choice}, we present the Lyapunov function that we will use 
for each system indexed by $r$, and derive some basic properties of this Lyapunov function. 
In particular, we will see that the Lyapunov function can be thought of as an approximation to 
the 
linear
objective, $\sum_{\bk \in \ck} x_{\bk}$.
We also point out the main difficulties in establishing the asymptotic optimality of GRAND($Z^p$) here. 
In Subsection \ref{sec-conc-bounds-stationary}, we show that some conditions hold for the process 
$\bx(\cdot)$ in the stationary regime, with high probability for large $r$. 
Subsection \ref{sec-lyap-drift} is the key subsection, where we introduce an ``artificial process,'' such that
the conditions considered in Subsection \ref{sec-conc-bounds-stationary} are ``enforced,'' i.e. they hold 
with probability $1$ (as opposed to  `with high probability,' as for the actual process).
For this artificial process, we derive high-probability estimates on the change of the Lyapunov function.
An informal ``roadmap'' is provided at the beginning of Subsection \ref{sec-lyap-drift} for this derivation. 
We conclude the proof of Theorem \ref{thm-zp} in Subsection \ref{sec-complete-thm-zp}, where we
use the high-probability bounds in Subsection \ref{sec-conc-bounds-stationary}
to show that, with high probability, the artificial and the actual processes coincide. This in turn allows us
to obtain high-probability estimates on the change of the Lyapunov function under the actual process.

Throughout this section,  
we drop superscript $r$ in notation -- such as
that appears in $\bx^r(t)$ -- pertaining to processes with parameter $r$. 

\subsection{A Lyapunov Function and Some Basic Properties} 
\label{sec-lyap-choice}

We will use the Lyapunov function $L^{(a)}$ for the proof of Theorem \ref{thm-zp}, 
where the parameter $a$ depends on the system scale $r$ as $a = r^{p-1}$. 
Correspondingly, $b$ will also depend on $r$, with $b=-\log a = (1-p) \log r$.
When it is clear from the context that we are working with $L^{(a)}$ for a given fixed $r$
(and corresponding $a=r^{p-1}$), we 
often drop the superscript $(a)$.

Let us now provide an important remark regarding 
the Lyapunov function $L^{(a)}$ with $a = r^{p-1}$, 
and the GRAND($Z^p$) algorithm. 
On the one hand, for any given system scale parameter $r$, $L^{(a)}$ 
has a fixed parameter $a=r^{p-1}$.
Furthermore, the expressions \eqn{eq-xi-is-derivative} and \eqn{eq-L-deriv} 
for $\Xi(\bx)$, the time derivative of $L^{(a)}$, 
make use of the {\em convention} \eqn{eq-convention-x0}, 
where $a = r^{p-1}$. 
On the other hand, GRAND($Z^p$)
itself does {\em not} use or need to know the parameter $r$,
since the actual 
number of zero-servers that it uses at time $t$ is $X_{\bZero}(t)=Z^p(t)$, 
which only uses the knowledge of $Z(t)$, the total number of customers at time $t$.

At this point, we can highlight the key technical difficulty of the analysis of GRAND($Z^p$), 
compared to that of GRAND($aZ$) in \cite{StZh2013}. Under GRAND($aZ$), the fluid-scaled quantities
$x^r_{\bk}$ are $O(1)$ for {\em all} configurations $\bk$. 
Consequently, as $r\to\infty$, the fluid-scaled
process $\bx^r(\cdot)$ converges to a well-defined fluid limit $\bx(\cdot)$, 
for any well-defined limiting initial state $\bx(0)$.
The Lyapunov function $L^{(a)}$, with fixed $a > 0$, is then used in \cite{StZh2013} to show the convergence
of all fluid trajectories $\bx(t)$ to the point $\bx^{*,a}$, as $t\to\infty$, which is the key part of the analysis 
of  GRAND($aZ$). Therefore, in \cite{StZh2013}, 
it sufficed to work with fluid limits and derivatives
of the Lyapunov function along fluid limit trajectories.

In this paper, to analyze GRAND($Z^p$), we use  $L^{(r^{p-1})}$ as a Lyapunov function for a system with given $r$.
The key difficulty is that under GRAND($Z^p$), some of the fluid-scaled $x^r_{\bk}$ are $o(1)$; 
for example, $x^r_{\bZero}=O(r^{p-1})$. 
A further complication 
is that different $x^r_{\bk}$ are on {\em different
scales} with respect to $r$; namely, they are $O(r^{s(\bk)})$ with the power $s(\bk)$ 
depending on $\bk$, $s(\bk) \in (0, 1]$. 
Because of the multi-scale system dynamics, 
the conventional fluid limit is not sufficient 
to establish a negative drift of the Lyapunov function.
Moreover, the use of
other, {\em local}, fluid limits (obtained under different space/time scalings)
also does not appear to be particularly useful, because, again,  $x^r_{\bk}$ evolve on different scales.
To overcome this difficulty, in Section~\ref{sec-lyap-drift}, we use 
direct (and more involved)
probabilistic
estimates of the Lyapunov function
increments, which are based on martingale theory, and which do not involve fluid or local fluid limits.

The following lemmas state some elementary properties related to the Lyapunov function $L^{(a)}$ with $a = r^{p-1}$. 
\begin{lem}
\label{lem-close-to-opt}
Let $\veps \in \left(0, \frac{1}{2}\right)$. 
Consider a sequence      
of points $\bx^{\circ,a}$ indexed by $a = r^{p-1}$, which satisfy 
\beql{eq-cond1111}
\left| \sum_{\bk\in \ck} k_i x_{\bk}^{\circ,a} - \rho_i\right| 
\le \frac{r^{-1/2+\veps}}{I}, ~~\forall i, \mbox{ for sufficiently large } r, 
\end{equation}
and 
$$
\chi_{\bk,\bk',i}\left(\bx^{\circ,a}\right) \to 0 ~~ \mbox{for all pairs of edges}~(\bk,i), (\bk',i),
$$
where we recall the definition of $\chi_{\bk,\bk',i}$ in \eqn{eq-diff-log}. 
Then, as $r\to\infty$, $L^{(a)}(\bx^{\circ,a}) \to L^{*}$,
where $L^{*}$ is the optimal value of (LP) defined by
\eqn{eq-opt}--\eqn{eq-non-neg}.
\end{lem}

{\em Proof.} For each $r$, and the corresponding $a=r^{p-1}$, denote 
$$\nu^{*,a}_i = 1 - \frac{\log(x^{\circ,a}_{\be_i})}{\log a}, ~~\forall i \in \ci.$$ 
Define $\bx^{*,a} = \left\{x_{\bk}^{*,a}, \bk \in \ck\right\}$ via product
form \eqn{eq-grand-product} corresponding to these $\nu^{*,a}_i$:
\beql{eq-grand-product-copy}
x^{*,a}_{\bk} = \frac{a}{c_{\bk}} \exp \left[b \sum_i k_i \nu_i^{*,a}\right] = \frac{1}{c_{\bk}} a^{1-\sum_i k_i \nu_i^{*,a}} , ~~ \bk \in \ck.
\end{equation}
Note that for any $r$ (and correspondingly, $a = r^{p-1}$), 
$\bx^{*,a}$ and $\bnu^{*,a}$ need not be the optimal primal and dual solutions to 
(CVX($a$)), defined by \eqn{eq-opt-grand}-\eqn{eq-non-neg-grand}, 
since $\bx^{*, a}$ need not belong to $\cx$.
Instead, point $\bx^{*,a}$ is the optimal solution for the problem
\begin{flalign}
 & &  \text{minimize } & L^{(a)}(\bx) & \\
&&\mbox{subject to } & \sum_{\bk\in\ck} k_i x_{\bk} = \sum_{\bk\in\ck} k_i x_{\bk}^{*,a}, ~~\forall i\in \ci, & \label{eq-equality-star}\\
&& & \quad \quad \quad x_{\bk} \geq 0, ~~\forall \bk \in \ck. 
\end{flalign}

We now claim that for all $\bk \in \ck$, 
\begin{equation}\label{eq-ratio-prod-form}
x^{\circ,a}_{\bk}/x^{*,a}_{\bk} \to 1 \mbox{ as } r \to \infty.
\end{equation} 
By convention, $x^{\circ,a}_{\bZero} = x^{*,a}_{\bZero} = a = r^{p-1}$. 
It is also easy to check that for all $i \in \ci$, $x^{\circ,a}_{\be_i} = x^{*,a}_{\be_i}$.
Consider a general configuration $\bk \in \ck$. 
For an edge $(\bk, i)$, we use $\tilde{\chi}_{(\bk,i)}$ as a shorthand for $\chi_{\be_i,\bk,i} (\bx^{\circ,a})$, 
and we have
\begin{eqnarray*}
\tilde{\chi}_{(\bk,i)} &=& \chi_{\be_i,\bk,i} (\bx^{\circ,a}) \\
&=& \log \left(k_i x^{\circ,a}_{\bZero} x^{\circ,a}_{\bk}\right) - \log \left(x^{\circ,a}_{\be_i} x^{\circ,a}_{\bk-\be_i}\right)\\
&=& \log \left(k_i a x^{\circ,a}_{\bk}\right) - \log \left(x^{*,a}_{\be_i} x^{\circ,a}_{\bk-\be_i}\right).
\end{eqnarray*}
By some simple algebraic manipulation, we have the recursion 
\begin{equation}\label{eq-circ-x-recursion}
x^{\circ, a}_{\bk} = \frac{e^{\tilde{\chi}_{(\bk,i)}}x^{*,a}_{\be_i}x^{\circ,a}_{\bk-\be_i}}{k_i a}
= \frac{1}{k_i} e^{\tilde{\chi}_{(\bk,i)}} a^{-\nu_i^{*,a}}x^{\circ,a}_{\bk-\be_i}.
\end{equation}
Consider a path that connects configuration $\bk$ and $\bZero$ using 
the following sequence of edges: the first $k_1$ edges are 
$(\bk, 1), (\bk-\be_1, 1), \cdots, (\bk - (k_1-1) \be_1, 1)$, 
followed by $k_2$ edges of the form 
$(\bk - (k_1-1) \be_1, 2), (\bk - (k_1-1) \be_1 - \be_2, 2), \cdots, (\bk - (k_1-1) \be_1 - (k_2 -1)\be_2, 2)$, 
etc. Thus, the path uses $\sum_i k_i$ edges to connect $\bk$ with $\bZero$. 
Using $\tilde{e}$ to denote a generic edge that belongs to this path, and using the recursion \eqn{eq-circ-x-recursion}, 
it is easy to see that 
\[
x^{\circ, a}_{\bk} = \frac{1}{c_{\bk}} \left[\prod_{\tilde{e}} \exp\left(\tilde{\chi}_{\tilde{e}}\right)\right] a^{1-\sum_i k_i \nu_i^{*,a}} 
= \left[\prod_{\tilde{e}} \exp\left(\tilde{\chi}_{\tilde{e}}\right)\right] x^{*,a}_{\bk}.
\]
Thus, 
\[
\frac{x^{\circ, a}_{\bk}}{x^{*,a}_{\bk}} = \left[\prod_{\tilde{e}} \exp\left(\tilde{\chi}_{\tilde{e}}\right)\right] \to 1,
\]
since $\tilde{\chi}_{\tilde{e}} \to 0$ for all edges $\tilde{e}$. This completes the proof of the claim. 

By \eqn{eq-ratio-prod-form}, we have 
$$
\|\bx^{\circ,a} - \bx^{*,a}\| \to 0, ~~r\to\infty.
$$
In addition, using condition \eqn{eq-cond1111}, we have 
that for each $i \in \ci$, the term $\sum_{\bk\in\ck} k_i x_{\bk}^{*,a}$ on the RHS of
the equality constraint \eqn{eq-equality-star} 
converges to $\rho_i$ as $r\to\infty$.
Furthermore, recall that $\left|L^{(a)}(\bx) - \sum_{\bk \in \ck} x_{\bk}\right| \to 0$, uniformly on $\bx$ from a bounded set.
Therefore, we must have $L^{(a)}(\bx^{*,a}) \to L^*$.
$\Box$

Lemma~\ref{lem-close-to-opt} implies the following property that we will actually use.

\begin{lem}
\label{lem-close-to-opt2}
For any $\gamma>0$, there exists $\delta_1>0$ such that, for all sufficiently large $r$ and $a=r^{p-1}$,
condition \eqn{eq-cond1111} and the condition that 
$$
|\chi_{\bk,\bk',i}| \le \delta_1 ~~ \mbox{for all pairs of edges}~(\bk,i), (\bk',i)
$$
imply
$$
L^{(a)}(\bx) - L^* \le \gamma.
$$
\end{lem}

\subsection{High Probability Steady-State Bounds on $\bx(\cdot)$}\label{sec-conc-bounds-stationary}
\label{subsec-wp-bounds}

In this subsection we show that certain conditions hold under GRAND($Z^p$)
in steady-state with high probability, converging to $1$ as $r\to\infty$. 
Moreover, these conditions hold over time intervals with length increasing with $r$ as $r^\alpha$, $\alpha>0$.

From now on in the paper, we fix $p$ satisfying \eqn{eq-cond-p}, and a corresponding $s$, satisfying
$7/8 < s < 1- \kappa(1-p)$. 

We often refer to the following conditions, at a given time $t\ge 0$, with some constants 
$c>0$ and $\veps>0$:
\beql{eq-cond1}
|Z(t)/r - 1| \le r^{-1/2+\veps};
\end{equation} 
\beql{eq-cond11}
\left| \sum_{\bk} k_i x_{\bk}(t) - \rho_i \right|
\le \frac{r^{-1/2+\veps}}{I}, ~~\forall i;
\end{equation} 
\beql{eq-cond2}
x_{\bk}(t) \ge c r^{s-1}, ~~ \forall \bk.
\end{equation}
It is easy to see that \eqn{eq-cond11} implies \eqn{eq-cond1}.

\begin{lem}
\label{lem-p1}
Let $\alpha > 0$ and $\veps > 0$. 
Consider our system in the stationary regime for each $r$. Then,
$$
\pr\left\{\mbox{condition \eqn{eq-cond11} holds for all} ~t\in [0,r^\alpha]  \right\} \to 1, ~~r\to\infty.
$$
\end{lem}

The proof of Lemma~\ref{lem-p1} is provided in Appendix~\ref{proof-property-1}. 
The statement of Lemma \ref{lem-p1} is very intuitive, because we know that in the stationary regime, 
$Y_i(t)$ has Poisson distribution with mean $\rho_i r$ for any $t$. 
Since $Y_i(t) = \sum_{\bk\in \ck}k_i X_{\bk}(t)$, 
condition \eqn{eq-cond11} clearly holds for any given $t$. 
The proof shows that, in fact, with high probability, \eqn{eq-cond11} holds on any $r^{\alpha}$-long time interval.

\begin{lem}
\label{lem-p2}
Let $\alpha > 0$.
Consider our system in the stationary regime for each $r$. Then, there exists $c>0$ such that 
$$
\pr\{\mbox{condition \eqn{eq-cond2} holds for } ~t\in [0,r^\alpha]  \} \to 1, ~~r\to\infty.
$$
\end{lem}

The proof of Lemma~\ref{lem-p2} is provided in Appendix~\ref{proof-property-2}, 
where we establish a stronger property: 
with high probability, there exists some $c > 0$ such that 
for all $\bk \in \ck$ and for all $t \in [0, r^{\alpha}]$, 
$x_{\bk}(t) \geq cr^{s(\bk)-1}$, where $s(\bk) = 1 - \left(\sum_i k_i+1\right)(1-p) \geq 1 - \kappa(1-p) > s$.
Here we provide a sketch of the argument used to establish this stronger property. 
By Lemma~\ref{lem-p1}, with high probability,
the fluid-scaled total number of customers, $Z(t)/r$, is in steady state close to $1$ for all $t\in [0,r^\alpha]$.
Therefore, $x_{\bZero}$ is close to $Z^p/r \approx r^{p-1}$, and the property 
should hold
for $\bk=\bZero$. Then, we use an 
induction argument. Fix configuration $\bk \neq \bZero$, and suppose the property 
is true for all configurations $\bk' \ne \bk$, $\bk' \le \bk$. 
We argue that it should then hold
for $\bk$ 
and some $c$ (possibly rechosen to be smaller), as follows.
 Pick $\bk'=\bk-\be_i$ for some $i$; such $\bk'$ exists since $\bk \neq \bZero$. 
We account separately for transitions that increase $x_{\bk}$ and those that decrease $x_{\bk}$. 
First, type-$i$ arrivals along edge $(\bk, i)$ increase $x_{\bk}$, 
and take place at (unscaled) rate $\lambda_i r x_{\bk'}/x_{(i)}$. 
Furthermore, this arrival process 
is lower bounded by a Poisson process of (unscaled) 
rate $c_1 r \left(r^{s(\bk')}/r\right) = c_1 r^{s(\bk')}$, 
for a constant $c_1>0$, 
because $X_{(i)}/r \le (X_{\bZero} + Z)/r$ is essentially upper bounded by a constant, 
and $x_{\bk'} \geq c (r^{s(\bk')}/r)$ by the induction hypothesis. 
Second, transitions that would decrease $x_{\bk}$ consist of departures from $\bk$ 
and arrivals into $\bk$. {\em When $x_{\bk}(t) \le c r^{s(\bk)-1}$}, 
the (unscaled) rates of these transitions are upper bounded by 
$c_2 c r^{s(\bk)}$ and $c_3 c r (r^{s(\bk)}/r^p) = c_3 c r^{s(\bk) + 1 -p}$ 
(because $X_{(i)}\geq X_{\bZero} \approx r^p$), respectively, 
for some constants $c_2>0$ and $c_3>0$. Thus, 
the transitions that would decrease $x_{\bk}$ are upper bounded by a Poisson process with (unscaled) rate 
$c_2 c r^{s(\bk)} + c_3 c r^{s(\bk) + 1 -p} \leq c_4 c r^{s(\bk) + 1 -p} = c_4 c r^{s(\bk')}$, 
for some $c_4>0$. 
Therefore, if we choose constant $c$ to be sufficiently small, then
when $x_{\bk}(t) \le c r^{s(\bk)}/r$, the rate of transitions that increase $x_{\bk}$ dominates 
the rate of transitions that decrease it, at least by some factor greater than $1$. 
Thus, with high probability, $x_{\bk}(t)$ ``should'' stay above $c r^{s(\bk)-1}$.

\subsection{Lyapunov Function Drift Estimates for an Artificial Process}
\label{sec-lyap-drift}

We will start with an informal overview of this subsection, and how its results will be used in the next one.

In this subsection, we consider an {\em artificial} process for which conditions \eqn{eq-cond1}-\eqn{eq-cond2} 
are ``enforced'' over a $r^{s-1}$-long interval; the construction is such that the {\em original} process pathwise
coincides with the artificial one, as long as \eqn{eq-cond1}-\eqn{eq-cond2} hold.
For the artificial process, we obtain high probability estimates of the increment of a certain process $F(t)$ over an $r^{s-1}$-long interval. The construction of $F(t)$ is such that, {\em if sample paths of the artificial and original processes coincide}, then $F(t)$ is equal to the increment of the Lyapunov function $L(\bx(t))$.
(Recall that $L=L^{(r^{p-1})}$ depends on $r$.)

The goal of this subsection is to obtain high probability estimates of $F(t)$. To do that, we use 
the martingale representation of $F(t)$, in which the compensator is $B(t)=\int_0^t AL(\bx(\xi)) d\xi$,
where operator $A$ is closely related to the generator of the artificial (and original) process. 
We then derive bounds on $|AL(\bx(\xi))-\Xi(\bx(\xi))|$, which allows us to obtain a bound 
on $|B(t) - \int_0^t \Xi(\bx(\xi))d\xi |$. 
We also obtain an upper bound on the quadratic characteristic (predictable quadratic variation) of the martingale $M(t)=F(t)-B(t)$, which, by Doob's inequality, gives us a high probability bound on $|M(t)|$. Combining the latter two bounds, we obtain one 
on $|F(t) - \int_0^t \Xi(\bx(\xi))d\xi|$, namely we can claim that it is ``small.'' Finally, we observe that $\Xi(\bx(\xi))\le 0$ 
for any $\bx(\xi)$, and
use Lemma \ref{lem-close-to-opt2} to show that $\Xi(\bx(\xi))$ is negative and bounded away from $0$ as long as  $\bx(\xi)$ is bounded away from the optimal set $\cx^*$ of (LP) in \eqref{eq-opt}--\eqref{eq-non-neg}. Combining all the above, we can claim that, with high probability, $F(t)$
has a negative (bounded away from zero) increment over a $r^{s-1}$-long interval, in which $\bx(t)$ ``stays away'' from the optimal set $\cx^*$.

In Subsection \ref{sec-complete-thm-zp} we will
use the fact that \eqn{eq-cond1}-\eqn{eq-cond2} hold for the original process with high probability, 
to couple the artificial and original processes in a way such that they coincide with high 
probability. Therefore, with high probability, the increment of the Lyapunov function $L(\bx(t))$ for the original process is equal to
$F(t)$, for which the estimates of this subsection apply. 

This ends the informal overview. Let us proceed with formal results.

In Subsections~\ref{sec-lyap-drift} and \ref{sec-complete-thm-zp} we use the following convention. By $C_g$ we denote a generic positive constant, which does not depend on the system parameters or
on the scaling parameter $r$; the value of $C_g$ may be different in different expressions, where it appears.

Consider a single $r^{s-1}$-long interval, and let $T = r^{s-1}$. 
The initial state is such that, at $t=0$, 
conditions \eqn{eq-cond1}, \eqn{eq-cond11}, and \eqn{eq-cond2} hold
for some constant $c>0$ and a small constant $\veps>0$, where the magnitude of $\veps$ 
will be specified later (at the end of this subsection).

Let $\widehat \cx= \widehat \cx^r$ denote the set of states $\bx$ that satisfy 
\eqn{eq-cond1}-\eqn{eq-cond2}. For $\bx \in \widehat \cx$, we have:
\beql{eq-deriv1}
\left| \frac{\partial}{\partial x_{\bk}} L(\bx) \right| \le C_g,
\end{equation}
\beql{eq-deriv2}
0 \le \left|\frac{\partial^2}{\partial x_{\bk}^2} L(\bx) \right| \le C_g [r^{s-1} \log r]^{-1}  \le C_g r^{1-s},
\end{equation}
for  all sufficiently large $r$. 
To see inequality \eqref{eq-deriv1}, first recall from \eqref{eq-L-partial} that
$\frac{\partial}{\partial x_{\bk}} L(\bx) = \frac{1}{b}  \log \left(\frac{c_{\bk}x_{\bk}}{a}\right) = \frac{1}{b}  \log \left(c_{\bk}x_{\bk}\right)+1$ for any $\bk$, 
since $b = -\log a$. Also, since $\bx \in \widehat \cx$ satisfies condition \eqref{eq-cond1}, 
$x_{\bk}\leq z \leq 2$ for sufficiently large $r$. Thus, \eqref{eq-deriv1} holds for all sufficiently large $r$. 
To see inequality \eqref{eq-deriv2}, note that $\frac{\partial^2}{\partial x_{\bk}^2} L(\bx) = \frac{1}{bx_{\bk}}$, 
$b = -\log a = (1-p) \log r$ (since $a = r^{p-1}$), and $x_{\bk} \geq c r^{s-1}$, by condition \eqref{eq-cond2}.

Our process $\bx(t)$ is a pure jump process -- the jumps are caused by customer arrivals and departures.
 For the rest of this subsection, we consider an artificial version
of the process, which we now describe. 
In this artificial process, for which, with some abuse, we keep the same notation $\bx(\cdot)$, 
some of the jumps (i.e., transitions) are not allowed to occur.
Specifically, if $\xi$ is a point in time 
when an arrival or departure takes place, 
we call it a point of {\em potential transition}.
If this potential transition keeps the state within conditions \eqn{eq-cond1}-\eqn{eq-cond2}, we allow it to occur, so this becomes an
actual transition for the artificial process. If the potential transition were to take us to a state violating \eqn{eq-cond1}-\eqn{eq-cond2}, we do {\em not} allow this transition
to actually occur, and keep the state unchanged. 

We now 
introduce some notation related to
the artificial process. 
For $\bx \in \widehat \cx$, let $\ct(\bx) = \ct^r(\bx)$ denote the set of all potential transitions out of it. 
 Any potential transition $\alpha \in \ct(\bx)$ uniquely identifies the ``from"-state, which is $\bx$, 
and the ``to''-state, which we denote by $\bx_{\alpha}$. Thus, the sets $\ct(\bx)$ for different $\bx$ are non-intersecting.
Let $\ct=\ct^r=\cup_{\bx \in \widehat{\cx}} \ct(\bx)$ 
be the set of all possible potential transitions from the states $\bx \in \widehat \cx$.
Note that for any $r$, the cardinalities of $\widehat \cx$ and $\ct$ are both finite.
Denote by $\widehat \ct(\bx) \subseteq \ct(\bx)$ 
the set of all allowed potential 
transition out of $\bx$, i.e., those for which $\bx_\alpha \in \widehat \cx$.
Let $\widehat \ct = \cup_{\bx \in \widehat \cx} \widehat\ct(\bx) \subseteq \ct$ denote
the set of all potential transitions that are allowed. Finally, for 
a potential transition $\alpha$, 
denote by $\beta_{\alpha}$ its (infinitesimal) rate.

We can and do construct the artificial process on the following probability space. 
Let the initial state $\bx(0) \in \widehat\cx$ be fixed. For each 
potential transition type $\alpha \in \ct$, there is an independent unit-rate Poisson process
$\Pi_\alpha(\cdot)$. The counting process of potential $\alpha$-transitions is
$$
H_\alpha(t) = \Pi_\alpha \left(  \int_0^t  \beta_\alpha I\{ \alpha \in \ct(\bx(\xi))\} d\xi \right).
$$
Denote by $\tau_\alpha(t)$ the set of time points where the potential transitions $\alpha$,
and the associated jumps of $H_\alpha(\cdot)$, occur in $[0,t]$. Denote by $\tau(t) = \cup_{\alpha \in \ct} \tau_\alpha(t)$
and $\widehat \tau(t) = \cup_{\alpha \in \widehat\ct} \tau_\alpha(t)$
the sets of all potential transition points and allowed (i.e., actual) transition points, respectively, in $[0,t]$. 
The state $\bx(t)$ at time $t$ is the state after the last allowed potential transition. More formally,
if $\max \widehat \tau(t) \in \tau_\alpha(t)$, then $\bx(t)=\bx_\alpha$,
where $\bx_\alpha$ is the ``to''-state of potential transition $\alpha$;
if $\widehat \tau(t)$ is empty, then $\bx(t)=\bx(0)$.
It is easy to see that, w.p.1, this construction uniquely determines the realization of the process,
including all potential transition points.

Consider the following pure-jump process:
$$
F(t) = \sum_{\alpha\in\ct} \sum_{s\in \tau_\alpha(t)} \Delta_\alpha 
\equiv \sum_{\alpha\in\ct} \Delta_\alpha H_\alpha(t),
$$
where
$$
\Delta_\alpha = L(\bx_\alpha) - L(\bx), 
$$
$\bx$ and $\bx_\alpha$ the ``from''- and ``to''-states of potential transition $\alpha$.
Note that the above definition of $F(t)$ has the sum over {\em all} potential transitions, both allowed and not allowed.
We observe for future reference that, if the artificial process realization is such that all potential transitions
in $[0,t]$  are allowed, i.e., they are actual transitions, then  $F(t)= L(\bx(t)) - L(\bx(0))$.

Denote 
$$
AL(\bx) = \sum_{\alpha \in \ct(\bx)} \beta_{\alpha} \Delta_\alpha
\equiv \sum_{\alpha \in \ct(\bx)} \beta_{\alpha} [L(\bx_{\alpha})-L(\bx)].
$$
Let us remark that for a process where all potential transitions are allowed, $AL$ is 
the process generator, if $L$ is within its domain. 
For our artificial process $\bx(\cdot)$, however, $AL(\bx)$ is just a formally defined expression.


We have
\beql{eq-mart}
F(t) = M(t) + B(t),
\end{equation}
where, as we will see shortly,
$$
B(t) = \sum_{\alpha\in \ct}  \Delta_\alpha \int_0^t  \beta_\alpha I\{ \alpha \in \ct(\bx(s))\} ds = \int_0^t AL(\bx(t)) dt
$$
is the compensator of $F(t)$ and $M(t)= F(t)-B(t)$ is martingale. We will use the following additional notation:
\beql{eq-mart1}
B(t) = \int_0^t AL(\bx(t)) dt = \sum_{\alpha\in \ct} B_\alpha(t),
\end{equation}
where
\beql{eq-mart11}
B_\alpha(t) = \Delta_\alpha \bar H_\alpha(t), ~~~
\bar H_\alpha(t) \equiv \int_0^t  \beta_\alpha I\{ \alpha \in \ct(\bx(s))\} ds;
\end{equation}
and 
\beql{eq-mart22}
M(t) = \sum_{\alpha\in \ct} M_\alpha(t),
\end{equation}
where
\beql{eq-mart2}
M_\alpha(t) = \Delta_\alpha \left[ \Pi_\alpha \left( \bar H_\alpha(t) \right)
- \bar H_\alpha(t) \right].
\end{equation}
Note also that 
$$
H_{\alpha}(t) = \Pi_\alpha \left( \bar H_\alpha(t) \right). 
$$
We have (from, e.g., Lemmas 3.1 and 3.2 in \cite{PTW2007}), that
each $M_\alpha(t), ~t\ge 0,$
is a martingale w.r.t. the filtration describing the history of the process up to time $t$. 
(This itself is a martingale,
there is no need to localize, as in \cite{PTW2007},
because in our case the total
transition rate and all jump sizes are uniformly bounded.) 
Moreover, its predictable and optional quadratic variations 
are, respectively,
$$
\langle M_\alpha \rangle (t) = \Delta_\alpha^2 \bar H_\alpha(t), ~~~ \left[ M_\alpha \right](t) = \Delta_\alpha^2 H_\alpha(t).
$$ 

We see that $M(\cdot)$ is also a martingale.
Furthermore, since w.p.1 all potential transition time points are distinct, 
we have
$$
\left[ M \right](t) = \sum_{\alpha \in \ct} \left[ M_\alpha \right](t)  =
\sum_{\alpha \in \ct} \Delta_\alpha^2 H_\alpha(t).
$$
Using this fact, and the fact that each $M_\alpha(\cdot)$ is a martingale, the same argument 
as in the proof of Lemma 3.1 in \cite{PTW2007} shows that 
$$
\langle M \rangle (t) = \sum_{\alpha \in \ct} \Delta_\alpha^2 \bar H_\alpha(t).
$$

We have the following deterministic upper bound on the predictable quadratic variation 
(quadratic characteristic):
\beql{eq-mart3}
\langle M \rangle (t) \le [C_g r] \left(\frac{C_g}{r}\right)^2 t, 
\end{equation}
where $C_g  r$ is a uniform upper bound on the total rate of potential transitions from any state, 
$C_g/r$ is a uniform upper bound on $\Delta_\alpha$. 

By Doob's inequality (see, e.g., Theorem 1.9.1 in \cite{Liptser_Shiryaev}), for any $\delta > 0$, 
\beql{eq-doob}
\pr\left\{\max_{0\le \xi \le t} |M(\xi)| \ge \delta\right\} \le \frac{\E [\langle M \rangle (t)]}{\delta^2}.
\end{equation}
In particular, for $\delta= \eta r^{3s-3-\veps}$ with 
some $\eta \in (0, 1/4)$, 
and $t=T=r^{s-1}$, we have
\beql{eq-doob2}
\pr\left\{\max_{0\le \xi \le T} |M(\xi)| \ge \eta r^{3s-3-\veps}\right\} \le \frac{[C_g \lambda r] \left(C_g/r\right)^2 r^{s-1}}{[\eta r^{3s-3-\veps}]^2} \le C_g  \frac{r^{4-5s+2\veps}}{\eta^2}.
\end{equation}

Our next goal is to estimate $|AL(\bx) - \Xi(\bx)|$, where $\Xi(\bx)$ is defined in \eqn{eq-L-deriv} and \eqn{eq-xi-is-derivative}. 
Denote by $\bar AL(\bx)$ the ``linear component'' of $AL(\bx)$, i.e. the $AL(\bx)$ computed with 
function $L$ replaced by its linearization at $\bx$. 
Specifically, 
\[
\bar AL(\bx) = \sum_{\alpha \in \ct(\bx)} \beta_{\alpha} (\bx_{\alpha} - \bx) \cdot \nabla L(\bx),
\]
where $\nabla L(\bx)$ is the gradient of $L$ at $\bx$. 
Since the difference between $L(\bx_{\alpha}) - L(\bx)$ and $(\bx_{\alpha} - \bx) \cdot \nabla L(\bx)$ 
is second order, using \eqn{eq-deriv2}, 
we have
\beql{eq-a-to-bar}
| AL(\bx) - \bar AL(\bx)| \le C_g r \cdot \frac{r^{1-s}}{\log r} \cdot \left(\frac{1}{r}\right)^2 < r^{-s},
\end{equation}
where the second inequality is true for $r>\exp (C_g)$.

Let us now consider 
the relation between $\bar AL(\bx)$ and $\Xi(\bx)$. 
Expression for $\bar AL(\bx)$ can be rewritten as 
\begin{equation}\label{eq-bar-AL}
\bar AL(\bx) = \sum_{(\bk,i) \in \cm} \left[\frac{\partial}{\partial x_{\bk}} L^{(a)}(\bx) - \frac{\partial}{\partial x_{\bk-\be_i}} L^{(a)}(\bx)\right]
\big[v_{\bk,i} - \tilde w_{\bk,i}\big],
\end{equation}
where $v_{\bk, i}$ are obtained from making the following modifications to $\tilde v_{\bk, i}$ 
in \eqn{eq-xi-is-derivative} and \eqn{eq-tilde-v}.
For each $i$, replace $\tilde{\lambda}_i$ by $\lambda_i$ (since $\lambda_i$ are the {\em actual} fluid-scaled arrival rates),
which changes the expression of $\tilde v_{\bk,i}$ in \eqn{eq-tilde-v} accordingly. 
Next, $x_{(i)}$ is defined as in \eqn{eq-convention-x0}, but $x_{\bZero} = a$ 
is replaced by the actual fluid-scaled number of zero-servers under GRAND($Z^p$), i.e., $x_{\bZero} = Z^p/r$, 
which further changes the value of $\tilde v_{\bk,i}$ in \eqn{eq-tilde-v}. 
We keep the convention $(\partial/\partial x_{\bZero}) L(\bx)=0$, for any $\bx$. 
This completes the description of the modifications. 
We see that {\em $\bar AL(\bx)$ is exactly equal to $\Xi(\bx)$ with each $\tilde v_{\bk,i}$ replaced 
by the corresponding $v_{\bk,i}$.}

We can now estimate $|\bar AL(\bx) - \Xi(\bx)|$. 
We know from \eqn{eq-cond11} that 
$$
\left| \tilde \lambda_i - \lambda_i \right| \le r^{-1/2 + \veps}.
$$
Furthermore, from \eqn{eq-cond1}, we have 
$$
\left|\frac{Z^p}{r^p} -1\right| \le 2 r^{-1/2+\veps}.
$$
Therefore, the modified (and {\em actual}) $x_{(i)}$ with $x_{\bZero}=Z^p/r$, and the value of $x_{(i)}$ in \eqn{eq-convention-x0}, 
assuming $x_{\bZero} = a=r^{p-1}$, are different by at most $C_g r^{p-1} r^{-1/2+\veps}$.
Summarizing these observations, we conclude that vectors $\tilde \bv=\{\tilde v_{\bk,i}\}$ and 
$\bv=\{v_{\bk,i}\}$ must be close, specifically
\beql{eq-cond11111}
\| \tilde \bv - \bv \| \le C_g r^{-1/2+\veps}.
\end{equation}

Using the expressions \eqn{eq-xi-is-derivative} for $\Xi(\bx)$ 
and \eqn{eq-bar-AL} for $\bar AL(\bx)$, we have
\begin{eqnarray}
|\Xi(\bx) - \bar AL(\bx)| &=& \left| \sum_{(\bk,i) \in \cm} \left[\frac{\partial}{\partial x_{\bk}} L^{(a)}(\bx) - \frac{\partial}{\partial x_{\bk-\be_i}} L^{(a)}(\bx)\right]
\big[v_{\bk,i} - \tilde v_{\bk,i}\big] \right| \\
&\leq & C_g\|\nabla L^{(a)}(\bx)\| \|\tilde{\bv}-\bv\| \\
&\leq & C_g r^{-1/2+\veps}, \label{eq-tilde-to-bar}
\end{eqnarray}
where the last inequality follows from \eqn{eq-deriv1} and \eqn{eq-cond11111}.

Summing \eqn{eq-a-to-bar} and \eqn{eq-tilde-to-bar}, 
recalling that $s > 7/8$, 
we obtain the estimate
$$
|AL(\bx) - \Xi(\bx)|\le |AL(\bx) - \bar AL(\bx)| + |\bar AL(\bx) - \Xi(\bx)|\le C_g r^{-1/2+\veps}.  
$$
Then,
\beql{eq-drift-error}
 \int_0^T |AL(\bx(t)) - \Xi(\bx(t))| dt  \le C_g r^{s-3/2+\veps}.
\end{equation}

Now we specify the choice of the constant 
$\veps>0$: it has to be small
enough to satisfy
\begin{equation}\label{eq-veps-veps2}
3s-3-\veps > s-3/2+\veps ~~(\mbox{or } ~\veps < s - 3/4), \mbox{ and }
\end{equation}
\begin{equation}\label{eq-veps2}
(-3s + 3 + \veps) + (4-5s + 2\veps) < 0 ~~(\mbox{or } ~\veps < (8s-7)/3). 
\end{equation}
Since $s > 7/8$, we have both $s - 3/4 > 0$ and $(8s-7)/3 > 0$. 
Thus, conditions \eqn{eq-veps-veps2} and \eqn{eq-veps2} are well-defined. 
Now, compare $\eta r^{3s-3-\veps}$, the term on the RHS of the event in bracket in \eqn{eq-doob2}, 
and $C_g r^{s-3/2+\veps}$, the term on the RHS of the error estimate in \eqn{eq-drift-error}. 
By condition \eqn{eq-veps-veps2}, we have $\eta r^{3s-3-\veps}> C_g r^{s-3/2+\veps}$ for all sufficiently large $r$.
Then, from \eqn{eq-mart}, \eqn{eq-doob2} and \eqn{eq-drift-error}, we obtain
\beql{eq-mart222}
\pr\left\{\max_{0\le t \le T} \left|F(t) - \int_0^t \Xi(\bx(t)) dt\right| \ge 2 \eta r^{3s-3-\veps}\right\}  \le C_g  r^{4-5s+2\veps}/\eta^2.
\end{equation}
Let us remark that to obtain \eqn{eq-mart222}, we only need 
$\veps$ to be positive and satisfy condition \eqn{eq-veps-veps2}. 
The condition \eqn{eq-veps2} is needed for the rest of the proof of Theorem \ref{thm-zp} 
in the next subsection.

We conclude this subsection with the estimates of $\Xi(\bx)$ and $\int_0^T \Xi(\bx(\xi)) d\xi$.
 By the remark after \eqn{eq-L-deriv}, we have $\Xi(\bx) \le 0$. 
Furthermore,
let any $\delta_1>0$ be fixed. Then,
by \eqn{eq-a35} and \eqn{eq-xi-is-derivative}, 
there exists $C=C(\delta_1)>0$ such that the following holds: 
if $|\chi_{\bk,\bk',i}| \ge \delta_1$ for at least one
pair of edges, 
then 
$$
\Xi(\bx) \le - C (1/\log r) [r^{s-1}]^2 < - r^{2s-2-\veps},
$$
for all sufficiently large $r$.
We obtain that, always,
\beql{eq-conc1}
\int_0^t \Xi(\bx(\xi)) d\xi \le 0, ~~ \forall  t \le T,
\end{equation}
and if $|\chi_{\bk,\bk',i}| \ge \delta_1>0$ for at least one pair of edges
in the entire interval $[0,T]$, then
\beql{eq-conc2}
\int_0^T \Xi(\bx(t)) dt  \le - C (1/\log r) [r^{s-1}]^2 r^{s-1} < - r^{3s-3-\veps},
\end{equation}
for all sufficiently large $r$.

\subsection{Completing the Proof of Theorem \ref{thm-zp}}
\label{sec-complete-thm-zp} 

We start with a comment about the probability space constructions.
The construction that we will use in this section does {\em not}
have to be same as those used in Subsection~\ref{sec-lyap-drift} or Appendices~\ref{proof-property-1}-\ref{proof-property-2}.
The probability spaces used in different parts of the paper can be different,
as long as we use them to obtain statements in the form of probability estimates, which is the case here.
Obviously, the probability estimates are valid regardless of the probability space construction used to derive them.


Let the constants 
$\veps>0$, $\eta > 0$ and $c>0$ be those chosen above in Section~\ref{sec-lyap-drift}. 
More specifically, 
$\veps > 0$ satisfies conditions \eqn{eq-veps-veps2} and \eqn{eq-veps2}, $\eta \in (0, 1/4)$, 
and $c>0$ is such that Lemma \ref{lem-p2} holds. 
Fix $C' >0$. (The exact choice of $C'$ will be specified later.) 
Consider the stationary version of the original process on $C' r^{-3s+3+\veps}$ consecutive $r^{s-1}$-long intervals,
starting time $0$. (I.e., the distribution at initial time $0$ the stationary distribution of the process.)
These $r^{s-1}$-long intervals we will call {\em subintervals}.
The proof of Theorem \ref{thm-zp} will be completed if we prove the following 

{\em Assertion.} For any fixed $\gamma>0$ and any fixed subsequence of $r$, there exists a further subsequence of $r$,
along which, w.p.1, for all sufficiently large $r$ the following properties hold for the original process:\\
(a) properties \eqn{eq-cond1}-\eqn{eq-cond2} hold at all times within all $C' r^{-3s+3+\veps}$ subintervals;\\
(b) at the end of the last subinterval, $\sum_{\bk \in \ck} x_{\bk} - L^* < \gamma$.

To prove this assertion, fix any $\gamma>0$ and any subsequence of $r$. 
In addition to the original process, consider the artificial process, coupled to it as follows.
The initial state of the artificial process is equal (w.p.1) to that of the original one.
If the initial state (of both processes) satisfies \eqn{eq-cond1}-\eqn{eq-cond2}, then
the artificial process evolves as it is defined, and it is coupled to be equal to the original process until the first time 
when \eqn{eq-cond1}-\eqn{eq-cond2} is violated (for the original process).
By convention, if the initial state (of both processes) violates \eqn{eq-cond1}-\eqn{eq-cond2}, then
the artificial process is ``frozen'', i.e. remains equal to the initial state at all times.

Let us focus on the artificial process, and
apply the estimate \eqn{eq-mart222} to each subinterval. 
(When \eqn{eq-mart222} is applied to a given subinterval, the time is shifted so that $t=0$ is the beginning of that subinterval.) Then, by \eqn{eq-mart222} and a simple union bound, 
we see that the probability that the event in the brackets in the LHS of \eqn{eq-mart222} holds for at least one
of the subintervals is upper bounded by
$$
C_g  C' r^{-3s+3+\veps} r^{4-5s+2\veps}/\eta^2 = C_g  C' r^{7-8s+3\veps}/\eta^2.
$$
Since $\veps$ satisfies \eqn{eq-veps2}, $C_g  C' r^{7-8s+3\veps}/\eta^2 \to 0$ as $r \to \infty$. 
Consider a further subsequence of $r$, increasing fast enough, e.g. $r=r(n) \ge e^n$,
so that the sum of these probabilities is finite.
 Then, for the artificial process,
by Borel-Cantelli lemma, 
w.p.1, for all large $r$, the condition
\begin{equation}\label{eq-F-AL-diff}
\max_{0\le t \le T} \left|F(t) - \int_0^t \Xi(\bx(t)) dt\right| < 2 \eta r^{3s-3-\veps}
\end{equation}
holds simultaneously for all $C' r^{-3s+3+\veps}$ subintervals; 
furthermore, we have \eqn{eq-conc1}-\eqn{eq-conc2} for all these subintervals simultaneously.

By Lemmas~\ref{lem-p1}  and \ref{lem-p2}, and Borel-Cantelli lemma, we can choose a further subsequence of $r$, along which
the original process is such that w.p.1, for all large $r$, conditions \eqn{eq-cond1}-\eqn{eq-cond2} 
hold on all $C' r^{-3s+3+\veps}$ subintervals, and therefore the 
artificial process and the original process coincide.
Therefore, along this subsequence, 
w.p.1, for all large $r$,
$F(t)$ (which we defined for the artificial process) is
equal to the increment of $L$, $F(t) = L(\bx(t)) -L(\bx(0))$, for the original process,
for all subintervals simultaneously, and we also have
\eqn{eq-conc1}-\eqn{eq-conc2} for all subintervals simultaneously. 

Then, w.p.1, for all large $r$, the following occurs for the original process.
If at the beginning of a subinterval, $\Delta L \equiv L - L^* \ge \gamma$,
then either condition $\Delta L \le \gamma$ is ``hit'' within the subinterval, or,
at the end of the subinterval $\Delta L$ is smaller by at least 
$r^{3s-3-\veps}/2 > 0$.
This follows from \eqn{eq-conc2}, \eqn{eq-F-AL-diff}, condition $\eta < 1/4$,
and the fact that (by Lemma~\ref{lem-close-to-opt2}) 
$\Delta L \ge \gamma$ implies that
$|\chi_{\bk,\bk',i}| \ge \delta_1 >0$
for some $\delta_1=\delta_1(\gamma)$ that (just like $\gamma$) 
 does not depend on $r$.
In addition, if at the beginning of a subinterval or any other point in it $\Delta L \le \gamma$,
then at the end of the subinterval $\Delta L < 2\gamma$.

Note that w.p.1, for all large $r$, at the beginning of the first subinterval,
$L(\bx(0)) - L^* \le C''$ for some constant $C''>0$ independent of $r$. (Recall that $L(\bx) = L^{(r^{p-1})}(\bx)$,
and it converges to $\sum_{\bk\in\ck} x_{\bk}$ uniformly on compact sets.) 
We now specify the choice of $C'$: it is any constant satisfying $C' > 2 C''$.
Given this choice, we see that (w.p.1, for all large $r$) condition $\Delta L \le \gamma$ is 
in fact ``hit'' within one of the subintervals; this in turn implies that 
 $\Delta L < 2\gamma$ must hold at the end of the last subinterval.
 Recall again that $L(\bx) = L^{(r^{p-1})}(\bx)$ converges to $\sum_{\bk\in\ck} x_{\bk}$ uniformly on compact sets.
 We finally obtain that, w.p.1, for all large $r$, $\sum_{\bk \in \ck} x_{\bk} - L^* < 3\gamma$
 at the end of the last subinterval, which (by rechoosing $\gamma$) completes the proof of the assertion,
 and of Theorem~\ref{thm-zp}. $\Box$
 
\section{Discussion}
\label{sec-discuss}

This paper continues the line of work, originated in \cite{StZh2013}, which shows that, surprisingly, extremely simple dynamic placement (packing) algorithms,
such as GRAND, can be asymptotically optimal. We show that GRAND($Z^p$) algorithm, which is as simple as GRAND($aZ$) in \cite{StZh2013},
is asymptotically optimal in a stronger sense than  GRAND($aZ$). The analysis of GRAND($Z^p$) is substantially more involved technically than that of 
 GRAND($aZ$), because it cannot be reduced to the analysis of fluid limits and/or local fluid limits. 

We believe that 
our main result, Theorem~\ref{thm-zp}, holds for any $p \in (0, 1)$, without the additional condition \eqn{eq-cond-p} 
that requires $p$ to be {\em sufficiently close} to $1$.
This additional condition is needed for our technical approach to work, and 
 is probably just technical. Removing or relaxing the additional condition on $p$ 
is interesting and important from both practical and methodological point of view, and
may be a subject of future work.

\newpage
\appendix
\begin{center}
{\Large\textbf{Appendix}}
\end{center}

\section{Proof of Lemma~\ref{lem-p1}}
\label{proof-property-1}
We start with a general outline.
First, in Subsection \ref{p1-prelim}, we provide some useful concentration inequalities 
of Poisson random variables, together with some notation and terminology that will be useful 
for both the proofs of Lemmas~\ref{lem-p1} and \ref{lem-p2}. 
The actual proof of Lemma~\ref{lem-p1} is in Subsection \ref{proof-p1}; 
its outline is as follows. First, in steady state, property \eqref{eq-cond11} 
holds with high probability (w.h.p) for each $t$, since $Y_i(t) = \sum_{\bk} k_i X_{\bk}(t)$ is a Poisson 
random variable with mean $\lambda_i r$. 
To show that w.h.p. property \eqref{eq-cond11} holds uniformly over a polynomially long 
interval $[0, r^{\alpha}]$, we divide the interval $[0, r^{\alpha}]$ into shorter sub-intervals 
of length $r^{-1/2}$, and show that w.h.p. property \eqref{eq-cond11} holds uniformly over each 
sub-interval. In fact, the probability of \eqref{eq-cond11} {\em not} holding uniformly over a sub-interval decays fast enough with $r$,
so that we can use a union bound to claim that w.h.p. \eqref{eq-cond11} holds for all sub-intervals simultaneously.
To obtain a single sub-interval bound, we first observe that at the beginning of a sub-interval, 
property \eqref{eq-cond11} holds w.h.p. by stationarity. 
Then, using the fact that $Y_i(\cdot)$ follows the evolution of an $M/M/\infty$ queue, 
we can upper bound the increase of $Y_i$ due to arrivals, 
and lower bound the decrease of $Y_i$ due to departures, 
over a subinterval. 
We have chosen the length $r^{-1/2}$ of each sub-interval small enough 
so that increase and decrease in $Y_i$ are small as well, 
and (by possibly rechoosing $\veps$) property \eqref{eq-cond11} 
would hold uniformly in a sub-interval, w.h.p.

\subsection{Preliminaries}\label{p1-prelim}
In this subsection, we recall some basic concentration inequalities 
for Poisson random variables. We then describe 
a convention that we will follow in this section and Appendix \ref{proof-property-2}.

\begin{prop}[Theorem 5.4, \cite{MU2005}]\label{prop:pois-conc}
Let $V$ be a Poisson random variable with mean $\nu$. 
If $x > \nu$, then 
\begin{equation}\label{eq:pois-upper0}
\pr(V\geq x) \leq \frac{e^{-\nu}(e\nu)^x}{x^x};
\end{equation}
and if $x < \nu$, then
\begin{equation}\label{eq:pois-lower0}
\pr(V \leq x) \leq \frac{e^{-\nu}(e\nu)^x}{x^x}.
\end{equation}
\end{prop}

A consequence of Proposition \ref{prop:pois-conc} 
is the following concentration bounds, 
which will be used extensively for the proofs of Lemmas \ref{lem-p1} and \ref{lem-p2}.
\begin{cor}\label{cor:pois-conc}
Let $V$ be a Poisson random variable with mean $\nu$, and let $w \in [0, \nu]$.
Then, 
\begin{equation}\label{eq:pois-upper}
\pr(V \geq \nu + w) \leq e^{-\frac{w^2}{4\nu}}, \mbox{ and }
\end{equation}
\begin{equation}\label{eq:pois-lower}
\pr(V \leq \nu - w) \leq e^{-\frac{w^2}{4\nu}}.
\end{equation}
\end{cor}
\begin{proof}[Proof sketch.]
Let $w \in [0, \nu]$. Then, by \eqref{eq:pois-upper0}, 
\[
\pr(V \geq \nu + w) \leq \frac{e^{-\nu} (e\nu)^{\nu+w}}{(\nu+w)^{\nu+w}} 
= \exp\left[w + (\nu + w)\log \frac{\nu}{\nu+w}\right]. 
\]
To establish \eqref{eq:pois-upper}, it suffices to prove that for all $w \in [0, \nu]$, 
\[
w + (\nu + w)\log \frac{\nu}{\nu+w} \leq -\frac{w^2}{4\nu}.
\]
The inequality can be established by observing that when $w = 0$, LHS = RHS = 0, 
and that for $w \in [0, \nu]$, 
\[
\frac{d}{dw} \left[w + (\nu + w)\log \frac{\nu}{\nu+w}\right] = \log \frac{\nu}{\nu + w} \leq  - \frac{w}{\nu+w} \leq - \frac{w}{2\nu} = \frac{d}{dw}\left[-\frac{w^2}{4\nu}\right].
\]
\eqref{eq:pois-lower} can be established in a similar way. We omit further details.
\end{proof}

To prove Lemmas \ref{lem-p1} and \ref{lem-p2}, we will often consider 
probability bounds for a sequence of events, indexed by $r$. 
To simplify exposition, we use the following notation. 
For a sequence of constants $\left\{C^{(r)}\right\}$ with $C^{(r)} \in [0, 1]$ 
for each $r$, 
we write 
\begin{equation}\label{eq:poly-r}
C^{(r)} \leq e^{-\text{poly}(r)} \left(\mbox{respectively, }C^{(r)} \geq 1 - e^{-\text{poly}(r)}\right),
\end{equation}
if there exists a positive constant $\gamma$ such that 
for all sufficiently large $r$, 
\[
C^{(r)} \leq e^{-r^{\gamma}} \left(\mbox{respectively, }C^{(r)} \geq 1 - e^{-r^{\gamma}}\right).
\]
It is useful to think of $C^{(r)}$ as probabilities of events indexed by $r$.
Note that if $\left\{C^{(r)}_1\right\}_r$ and $\left\{C^{(r)}_2\right\}_r$ are two sequences with 
$C^{(r)}_1 \leq e^{-\text{poly}(r)} \mbox{ and } C^{(r)}_2 \leq e^{-\text{poly}(r)}$, 
then we also have 
\[
C^{(r)}_1 + C^{(r)}_2 \leq e^{-\text{poly}(r)}. 
\]
Similarly, if $C^{(r)}_1 \geq 1 - e^{-\text{poly}(r)} \mbox{ and } C^{(r)}_2 \geq 1 - e^{-\text{poly}(r)}$, 
then 
\[
C^{(r)}_1 + C^{(r)}_2 \geq 1 - e^{-\text{poly}(r)}. 
\]
Finally, we also often use the expression ``with probability $1 - e^{-\text{poly}(r)}$'' 
to mean that the probability is at least $1 - e^{-r^{\gamma}}$, 
for some $\gamma > 0$ and for all sufficiently large $r$.

\subsection{Proof of Lemma \ref{lem-p1}} \label{proof-p1}
Our general strategy for establishing Lemma \ref{lem-p1} 
is to divide the interval $[0, r^{\alpha}]$ into 
shorter sub-intervals of length $r^{-1/2}$, 
and show that \eqn{eq-cond11} holds with high probability 
over each sub-interval. 
More specifically, let us define events
\begin{equation}\label{eq:event-j}
E^r_{j,i} = \left\{\sup_{t \in \left[(j-1)r^{-\frac{1}{2}}, j r^{-\frac{1}{2}}\right]} |Y^r_i(t) - \rho_i r| \geq \frac{r^{1/2 + \veps}}{I}\right\},
\end{equation}
for $i \in \ci$, $j \in \{ 1, 2, \cdots, r^{\alpha +\frac{1}{2}}\}$\footnote{We treat $r^{\alpha+\frac{1}{2}}$ as if it were guaranteed to be an integer. 
Similarly, in the sequel, whenever $Cr^{\beta}$ is used as an index for some constants $\beta > 0$ and $C > 0$, 
we treat it as an integer. 
Rounding them up or down to a nearest integer would overburden our notation, but would not affect our order-of-magnitude estimates.}, and for each $r$. 
Note that because for each $r$, the system is in the stationary regime, 
we have that for any $i \in \ci$, 
\begin{equation}\label{eq:stationary-equal-events}
\pr\left(E^r_{1,i}\right) = \pr\left(E^r_{2,i}\right) = \cdots = \pr\left(E^r_{r^{\alpha + 1/2},i}\right).
\end{equation}
Thus, if we can show that
\begin{equation}\label{eq:e-poly-r-sub-int}
\pr\left(E^r_{1,i} \right) \leq e^{-\text{poly}(r)}, \forall i \in \ci,
\end{equation}
then by \eqref{eq:stationary-equal-events} and \eqref{eq:e-poly-r-sub-int}, 
\[
\pr\left(\cup_{1\leq j \leq r^{\alpha+1/2}, i \in \ci} E^r_{j, i}\right) \leq e^{-\text{poly}(r)} \to 0 \quad \text{as} \quad r \to \infty, 
\]
which establishes Lemma \ref{lem-p1} immediately.

It is now left to show that $\pr\left(E^r_{1, i}\right) \leq e^{-\text{poly}(r)}$ for each $i \in \ci$.

\begin{lem}\label{lem:p1-t}
Let $i \in \ci$, and consider the system indexed by $r$ in the stationary regime. Then, for sufficiently large $r$, 
\begin{equation}\label{eq:conc-t}
\pr\left(E^r_{1,i}\right) = \pr\left(\sup_{t \in [0, r^{-1/2}]} |Y^r_i(t) - \rho_i r| \geq \frac{r^{1/2 + \veps}}{I}\right)
\leq e^{-\text{poly}(r)}.
\end{equation}
\end{lem}
\begin{proof} For notational convenience, we drop the subscript $i$. 
The proof of Lemma \ref{lem:p1-t} consists of establishing 
(a) a probability tail bound of $|Y^r(0) - \rho r|$; and
(b) a probability tail bound of $\sup_{t \in [0, r^{-1/2}]} \left(Y^r(t) - \rho r\right)$; 
and (c) a probability tail bound of $\inf_{t \in [0, r^{-1/2}]} \left(Y^r(t) - \rho r\right)$.

(a) We first show that 
\begin{equation}\label{eq:conc-0}
\pr\left(|Y^r(0) - \rho r| \geq r^{1/2 + \veps/3}\right) \leq e^{-\text{poly}(r)}.
\end{equation}
Since $Y^r(0)$ is a Poisson random variable with mean $\rho r$, 
we apply the concentration bounds \eqn{eq:pois-upper} and \eqn{eq:pois-lower}. 
First by \eqn{eq:pois-upper}, 
we have that for sufficiently large $r$, $r^{1/2 + \veps/3} \leq \rho r$, and 
\[
\pr\left(Y^r(0) - \rho r \geq r^{1/2 + \veps/3}\right)
\leq \exp\left(-\frac{r^{1+2\veps/3}}{4\rho r}\right) 
= \exp\left(-\frac{r^{2\veps/3}}{4\rho}\right).
\]
Similarly, for sufficiently large $r$, 
\[
\pr\left(Y^r(0) - \rho r \leq -r^{1/2 + \veps/3} \right) 
\leq \exp\left(-\frac{r^{2\veps/3}}{4\rho}\right).
\]
Thus, by a simple union bound, for sufficiently large $r$,
\begin{equation}\label{eq:p1-part-a}
\pr\left(|Y^r(0) - \rho r| \geq r^{1/2 + \veps/3} \right) 
\leq 2\exp\left(-\frac{r^{2\veps/3}}{4\rho}\right) \leq e^{-\text{poly}(r)}.
\end{equation}
This completes part (a). 

(b) We then show that 
\begin{equation}\label{eq:conc-sup}
\pr\left(\sup_{t \in [0, r^{-1/2}]} \left(Y^r(t) - \rho r\right) \geq r^{1/2 + 2\veps/3}\right)
\leq e^{-\text{poly}(r)}.
\end{equation}
To establish \eqref{eq:conc-sup}, we use the following representation of the process $Y^r(\cdot)$. 
Note that the process $Y^r(\cdot)$ describes the steady-state evolution of a $M/M/\infty$ queueing system, 
so we can represent $Y^r(\cdot)$ as (see e.g., \cite{PTW2007}) 
\begin{equation}\label{eq:m/m/infty}
Y^r(t) = Y^r(0) + \Pi(\lambda r t) - \wt{\Pi}\left(\int_0^t \mu Y^r(s) ds\right),
\end{equation}
where $Y^r(0)$ is a Poisson random variable with mean $\rho r$, and 
$\Pi(\cdot)$ and $\wt{\Pi}(\cdot)$ are independent unit-rate Poisson processes
that are also independent from $Y^r(0)$. 

By \eqn{eq:m/m/infty}, we have that w.p.$1$, for all $t \in [0, r^{-1/2}]$, 
\[
Y^r(t) \leq Y^r(0) + \Pi(\lambda r t) \leq Y^r(0) + \Pi\left(\lambda r \cdot r^{-1/2}\right) 
= Y^r(0) + \Pi\left(\lambda r^{1/2}\right). 
\]
Using \eqn{eq:pois-upper} and the fact that $\Pi\left(\lambda r^{1/2}\right)$ is Poisson with mean $\lambda r^{1/2}$, we have that 
\begin{equation}\label{eq:pois-upper-bound}
\pr\left(\Pi\left(\lambda r^{1/2}\right) \geq 2\lambda r^{1/2} \right) 
\leq \exp\left(-\frac{\lambda r^{1/2}}{4}\right) \leq e^{-\text{poly}(r)}.
\end{equation}
For sufficiently large $r$, if $Y^r(0) + \Pi\left(\lambda r^{1/2}\right) \geq \rho_i r + r^{1/2 + 2\veps/3}$, then we must have
$Y^r(0) \geq \rho_i r + r^{1/2 + \veps/3}$ or $\Pi\left(\lambda r^{1/2}\right) \geq 2\lambda r^{1/2}$. Thus, 
\begin{eqnarray*}
\pr\left(Y^r(0) + \Pi\left(\lambda r^{1/2}\right) \geq \rho_i r + r^{1/2 + 2\veps/3}\right)
&\leq & \pr\left(Y^r(0) \geq \rho_i r + r^{1/2 + \veps/3} \mbox{ or } \Pi\left(\lambda r^{1/2}\right) \geq 2\lambda r^{1/2}\right) \\
&\leq & \pr\left(Y^r(0) \geq \rho_i r + r^{1/2 + \veps/3} \right) \\
& & + \pr\left(\Pi\left(\lambda r^{1/2}\right) \geq 2\lambda r^{1/2}\right) \\
&\leq & e^{-\text{poly}(r)}, 
\end{eqnarray*}
where the last inequality follows from \eqref{eq:conc-0} and \eqref{eq:pois-upper-bound}.

Since with probability $1$, for all $t \in [0, r^{-1/2}]$, 
\[
Y^r(t) \leq Y^r(0) + \Pi\left(\lambda r^{1/2}\right),
\]
it immediately follows that \eqn{eq:conc-sup} holds. This completes part (b).

(c) We now show that 
\begin{equation}\label{eq:conc-inf}
\pr\left(\inf_{t \in [0, r^{-1/2}]} (Y^r(t) - \rho r) \leq - r^{1/2 + 2\veps/3}\right)
\leq e^{-\text{poly}(r)}.
\end{equation}
Using the representation \eqref{eq:m/m/infty}, we have that 
with probability $1$, for all $t \in [0, r^{-1/2}]$, 
\[
Y^r(t) \geq Y^r(0) - \wt{\Pi}\left(\int_0^t \mu Y^r(s) ds\right) 
\geq Y^r(0) - \wt{\Pi}\left(\int_0^{r^{-1/2}} \mu Y^r(s) ds\right). 
\]
Consider events $F^r$ and $G^r$ defined to be 
\[
F^r = \left\{\sup_{t \in [0, r^{-1/2}]} Y^r(t) < 2 \rho r\right\}, \quad \text{and} \quad 
G^r = \left\{\wt{\Pi}\left(r^{1/2+\veps/3}\right) < 2 r^{1/2 + \veps/3}\right\}.
\]
For sufficiently large $r$, we have that under event $F^r$,  
\[
\int_0^{r^{-1/2}} \mu Y^r(s) ds < \mu r^{-1/2} \cdot (2\rho r) 
\leq r^{1/2 + \veps/3}.
\]
Thus, for sufficiently large $r$, we have that under the event $F^r \cap G^r$, 
\[
\wt{\Pi}\left(\int_0^{r^{-1/2}} \mu Y^r(s) ds\right) \leq \wt{\Pi}\left(r^{1/2 + \veps/3}\right) 
< 2 r^{1/2 + \veps/3}.
\]
By \eqn{eq:conc-sup}, we know that $\pr\left(F^r\right) \geq 1 - e^{-\text{poly}(r)}$, 
and using \eqn{eq:pois-upper}, we can easily show that 
$\pr\left(G^r\right) \geq 1 - e^{-\text{poly}(r)}$. 
Thus, we have $\pr\left(F^r\cap G^r\right) \geq 1 - e^{-\text{poly}(r)}$, 
and by considering the complement, we have
\begin{equation}\label{eq:pois-lower-bound}
\pr\left(\wt{\Pi}\left(\int_0^{r^{-1/2}}\mu Y^r(s) ds\right) \geq 2 r^{1/2 + \veps/3} \right) 
\leq e^{-\text{poly}(r)}.
\end{equation}
For sufficiently large $r$, if $Y^r(0) - \wt{\Pi}\left(\int_0^{r^{-1/2}}\mu Y^r(s) ds\right) \leq \rho_i r - r^{1/2 + 2\veps/3}$, 
then we have $Y^r(0) \leq \rho_i r - r^{1/2 + \veps/3}$ or $\wt{\Pi}\left(\int_0^{r^{-1/2}}\mu Y^r(s) ds\right) \geq 2 r^{1/2 + \veps/3}$. 
Thus, similar to the argument in part (b), we have that 
\begin{eqnarray*}
\pr\left(Y^r(0) - \wt{\Pi}\left(\int_0^{r^{-1/2}}\mu Y^r(s) ds\right) \leq \rho_i r - r^{1/2 + 2\veps/3}\right)
&\leq & \pr\left(Y^r(0) \leq \rho_i r - r^{1/2 + \veps/3} \right) \\
& & + \pr\left(\wt{\Pi}\left(\int_0^{r^{-1/2}}\mu Y^r(s) ds\right) \geq 2 r^{1/2 + \veps/3}\right) \\
&\leq & e^{-\text{poly}(r)}, 
\end{eqnarray*}
where the last inequality follows from \eqref{eq:conc-0} and \eqref{eq:pois-lower-bound}.
This establishes \eqn{eq:conc-inf}, and completes part (c). 

Combining \eqn{eq:conc-0}, \eqn{eq:conc-sup} and \eqn{eq:conc-inf}, 
we have 
\[
\pr\left(\sup_{t \in [0, r^{-1/2}]} |Y^r(t) - \rho r| \geq r^{1/2 + 2\veps/3}\right)
\leq e^{-\text{poly}(r)}.
\]
Since for sufficiently large $r$, $\frac{r^{1/2 + \veps}}{I} \geq r^{1/2 + 2\veps/3}$, we also have
\[
\pr\left(\sup_{t \in [0, r^{-1/2}]} |Y^r(t) - \rho r| \geq \frac{r^{1/2 + \veps}}{I}\right) 
\leq \pr\left(\sup_{t \in [0, r^{-1/2}]} |Y^r(t) - \rho r| \geq r^{1/2 + 2\veps/3}\right)
\leq e^{-\text{poly}(r)}.
\]
This establishes Lemma \ref{lem:p1-t}. 
\end{proof}

\section{Proof of Lemma~\ref{lem-p2}} \label{proof-property-2}

We already gave an informal sketch of this proof immediately after Lemma~\ref{lem-p2} statement,
at the end of Subsection~\ref{subsec-wp-bounds}.

For each $\bk \in \bar\ck$, define constant 
$s(\bk) = 1 - \left(1 + \sum_{i \in \ci} k_i\right) (1-p)$. 
To establish Lemma \ref{lem-p2}, we consider instead the condition 
\begin{equation}\label{eq:p2'}
X_{\bk}^r(t) \ge c r^{s(\bk)}
\end{equation}
for each $\bk \in \bar{\ck}$, at any time $t \geq 0$, and prove the following stronger result.
\begin{lem}\label{lem:property2'}
Let $\alpha > 0$. 
Consider our sequence of systems in the stationary regime, under the GRAND($Z^p$) algorithm. 
Then for each $\bk \in \bar{\ck}$, there exists a constant $c>0$, such that as $r \to \infty$, 
\begin{equation}\label{eq:prob-p2'}
\pr(\mbox{Condition \eqn{eq:p2'} holds for all } t\in [0, r^{\alpha}]) \geq 1 - e^{-\text{poly}(r)}.
\end{equation}
\end{lem}
To prove Lemma \ref{lem:property2'}, we use the following construction 
of the underlying probability space. 
For each $(\bk, i) \in \cm$, let $\Pi_{(\bk, i)}(\cdot)$ and $\wt{\Pi}_{(\bk, i)}(\cdot)$
be unit-rate Poisson processes. Furthermore, 
all $\Pi_{(\bk, i)}(\cdot)$ and $\wt{\Pi}_{(\bk, i)}(\cdot)$ are independent. 
We will use $\Pi_{(\bk, i)}(\cdot)$ 
as the driving processes for customer arrivals, 
and $\wt{\Pi}_{(\bk, i)}(\cdot)$ to drive departures. 
Furthermore, processes $\Pi_{(\bk, i)}(\cdot)$ and $\wt{\Pi}_{(\bk, i)}(\cdot)$ 
are all independent from the initial random system state $\bX^r(0)$.

More specifically, let $D^r_{(\bk, i)}(t)$ denote the total number
of departures along the edge $(\bk, i)$ in $[0, t]$. 
Then, 
\begin{equation}\label{eq:departure-ki}
D^r_{(\bk, i)}(t) = \wt{\Pi}_{(\bk, i)}\left(\int_0^t X^r_{\bk}(u) k_i \mu_i du\right).
\end{equation}
Similarly, let $A^r_{(\bk, i)}(t)$ denote the total 
number of arrivals along the edge $(\bk, i)$ in $[0, t]$. 
Under the GRAND algorithm, 
a type-$i$ customer that arrives at time $s$ 
is placed in a server of configuration $\bk-\be_i$ 
with probability $X^r_{\bk-\be_i}(s)/X^r_{(i)}(s)$. 
Then, we can write
\begin{equation}\label{eq:arrival-ki}
A^r_{(\bk, i)}(t) = \Pi_{(\bk, i)}\left(\int_0^t \lambda_i r \cdot \frac{X^r_{\bk-\be_i}(u)}{X^r_{(i)}(u)} du\right).
\end{equation}
Thus, for each $\bk \in \ck$, 
we can write 
\begin{eqnarray}
X^r_{\bk}(t) - X^r_{\bk}(0) &=& \left[\sum_{i: \bk-\be_{i} \in \bar{\ck}} A^r_{(\bk, i)}(t) + \sum_{i: \bk+\be_{i} \in \ck} D^r_{(\bk+\be_i, i)} (t)\right] \nonumber \\
& &- \left[\sum_{i: \bk+\be_{i} \in \ck} A^r_{(\bk+\be_{i}, i)}(t) + \sum_{i: \bk-\be_{i} \in \bar{\ck}} D^r_{(\bk, i)} (t)\right] \label{eq:dynamics1} \\
& \geq & \left[\sum_{i: \bk-\be_{i} \in \bar{\ck}} A^r_{(\bk, i)}(t)\right] 
- \left[\sum_{i: \bk+\be_{i} \in \ck} A^r_{(\bk+\be_{i}, i)}(t) + \sum_{i: \bk-\be_{i} \in \bar{\ck}} D^r_{(\bk, i)} (t)\right] \nonumber \\
& = & S_1(t) - S_2(t) - S_3(t), \label{eq:dynamics2}
\end{eqnarray}
where 
\begin{align}
S_1(t) & = \sum_{i: \bk-\be_i \in \bar{\ck}} A^r_{(\bk, i)}(t); \label{eq:S1}\\ 
S_2(t) & = \sum_{i: \bk+\be_i \in \ck} A^r_{(\bk+\be_i, i)}(t); \label{eq:S2} \\
S_3(t) & = \sum_{i: \bk-\be_i \in \bar{\ck}} D^r_{(\bk, i)} (t). \label{eq:S3}
\end{align}
The equality \eqn{eq:dynamics1} follows by accounting for all arrivals and departures that contribute 
to $X^r_{\bk}(t) - X^r_{\bk}(0)$, the change in $X^r_{\bk}$. 
For example, the term $\sum_{i: \bk-\be_i \in \bar{\ck}} A^r_{(\bk, i)}(t)$ 
accounts for arrivals along the edges $(\bk, i) \in \cm$, 
which increases the number of servers of configuration $\bk$.

\begin{proof}[Proof of Lemma \ref{lem:property2'}]
We prove Lemma \ref{lem:property2'} by induction on $\|\bk\|_1 = \sum_i k_i$. 

{\em Base case: $\|\bk\|_1 = 0$.} If $\|\bk\|_1 = 0$, 
then $\bk = \bZero$, and $s(\bZero) = 1 - (1-p) = p$. 
Then, we show that 
\[
\pr\left(X_{\bZero}^r(t) \ge \frac{1}{2} r^{p}, \ \forall t \in [0, r^{\alpha}]\right) \geq 1 - e^{-\text{poly}(r)}.
\]
To this end, note that $X_{\bZero}^r(t) = \lceil(Z^r(t))^p\rceil$ for all $t$. 
By Lemma \ref{lem-p1}, we know that for any $\veps > 0$, as $r \to \infty$, 
\[
\pr\left(|Y_i^r(t) - \rho_i r| \le \frac{r^{1/2+\veps}}{I}, \ \forall t \in [0, r^{\alpha}], \forall i \in \ci \right) 
\geq 1 - e^{-\text{poly}(r)}. 
\]
In particular, by setting $\veps = 1/4$, we have
\[
\pr\left(|Y_i^r(t) - \rho_i r| \le \frac{r^{3/4}}{I}, \ \forall t \in [0, r^{\alpha}], \forall i \in \ci \right) 
\geq 1 - e^{-\text{poly}(r)}. 
\]
For sufficiently large $r$, $\frac{r}{2^{1/p}} \leq r - r^{3/4}$. 
Since $Z^r(t) = \sum_{i \in \ci} Y_i^r(t)$ for all $t$,  we have
\begin{eqnarray*}
\pr\left(Z^r(t) \ge \frac{r}{2^{1/p}}, \ \forall t \in [0, r^{\alpha}]\right) 
& \geq & 
\pr\left(|Z^r(t) - r| \le r^{3/4}, \ \forall t \in [0, r^{\alpha}]\right) \\
& \geq & 
\pr\left(|Y_i^r(t) - \rho_i r| \le \frac{r^{3/4}}{I}, \ \forall t \in [0, r^{\alpha}], \forall i \in \ci \right) \\
&\geq & 1 - e^{-\text{poly}(r)}.
\end{eqnarray*}
This implies that 
\[
\pr\left(X_{\bZero}^r (t) \ge  \frac{1}{2}r^{p}, \ \forall t \in [0, r^{\alpha}]\right) \geq
\pr\left(Z^r(t) \ge \frac{r}{2^{1/p}}, \ \forall t \in [0, r^{\alpha}]\right)
\geq 1- e^{-\text{poly}(r)},
\]
establishing the base case. 

{\em Induction step.} Let $\ell > 0$. Suppose that for all $\bk'$ with $\|\bk'\|_1 \le \ell-1$, 
there exists $c' > 0$ such that 
\[
\pr\left(X_{\bk'}^r(t) \ge c' r^{s(\bk')},\ \forall t \in [0, r^{\alpha}]\right) \geq 1 - e^{-\text{poly}(r)}. 
\]
Let $\bk \in \ck$ be a configuration with $\|\bk\|_1 = \ell > 0$. 
Then, there exists $\iota \in \ci$ such that $k_{\iota} \ge 1$. Let $\tilde{\bk} = \bk - \be_{\iota}$
so that $\|\tilde{\bk}\|_1 = \ell - 1$. 
For notational convenience, write $s_1 = s(\tilde{\bk})$ 
and $s_2 = s_1 - (1 - p)$ so that $s_2 = s(\bk)$. 

By the induction hypothesis, we can assume that 
\begin{equation}\label{eq:ind-hypo}
\pr\left(X_{\tilde{\bk}}^r(t) \ge \tilde{c} r^{s_1},\ \forall t \in [0, r^{\alpha}]\right) \geq 1 - e^{-\text{poly}(r)} 
\end{equation}
for some $\tilde{c} > 0$. Furthermore, 
define $E^r$ to be the event 
\begin{equation}\label{eq:event-ind-hypo}
E^r = \left\{X_{\tilde{\bk}}^r(t) \ge \tilde{c} r^{s_1},\ \forall t \in [0, r^{\alpha}]\right\}. 
\end{equation}
Then by \eqref{eq:ind-hypo} and \eqref{eq:event-ind-hypo}, 
$\pr\left(E^r\right) \geq 1 - e^{-\text{poly}(r)}$.

Consider the interval $[0, r^{\alpha}+1]$, and divide it into $\frac{r^{\alpha}+1}{T}$ sub-intervals of length $T = r^{s_2-1 - \veps'}$ (namely, sub-intervals $[0, T], [T, 2T], [2T, 3T], \cdots$),  
where $0 < \veps' < s_2 - \frac{p}{2}$. 
We claim the following. 

{\em Claim 1.} There exist positive constants $c$ and $\bar{c}$
such that with probability $1 - e^{-\text{poly}(r)}$, for each $j \in \{1, 2, \cdots, \frac{r^{\alpha}+1}{T}\}$:
\begin{itemize}
\item[(i)] if $X^r_{\bk}((j-1)T) \leq 4c r^{s_2}$, 
then 
\[
X^r_{\bk}(jT) - X^r_{\bk}((j-1)T) \geq \bar{c} r^{2s_2 - p - \veps'}; 
\]
\item[(ii)] if $X^r_{\bk}((j-1)T) \geq 2c r^{s_2}$, 
then 
\[
\inf_{t \in [(j-1)T, jT]} X^r_{\bk}(t) \geq \frac{1}{2}X^r_{\bk}((j-1)T). 
\]
\end{itemize}
Assuming the validity of Claim 1, we now complete the rest of the induction step, and defer the proof of the claim. 

Suppose that Claim 1 is true. For each $r$, consider a sample path 
for 
which both statements (i) and (ii) of Claim 1 
hold for all $j \in \{1, 2, \cdots, \frac{r^{\alpha}+1}{T}\}$, 
and let $j_0$ be minimal such that 
\begin{equation}\label{eq:j_0}
X^r_{\bk}(j_0 T) \geq 4c r^{s_2}. 
\end{equation}
Then, for sufficiently large $r$, $j_0 \leq \frac{1}{T}$. If $X^r_{\bk}(0) \geq 4c r^{s_2}$, then $j_0 = 0 \leq \frac{1}{T}$. 
If $X^r_{\bk}(0) < 4c r^{s_2}$, then we also have $j_0\leq \frac{1}{T}$, since 
by comparing the threshold $4c r^{s_2}$ and increment $\bar{c} r^{2s_2 - p - \veps'}$ 
from $X^r_{\bk}(jT)$ to $X^r_{\bk}((j-1)T)$ in statement (i), 
we have 
\[
\frac{4c r^{s_2}}{\bar{c} r^{2s_2 - p - \veps'}} = \frac{4c}{\bar{c}}\cdot \frac{1}{r^{s_2 - p -\veps'}} 
= \frac{4c}{\bar{c}}\cdot \frac{r^{p-1}}{T} \leq \frac{1}{T}.
\]
Next, claim for any $j \in \left\{j_0, j_0 + 1, \cdots, \frac{r^{\alpha} +1}{T}\right\}$, 
we have 
\begin{equation}\label{eq:prop2'-end-pt}
X^r_{\bk}(jT) \geq 2c r^{s_2},
\end{equation} 
and we establish this claim by induction (note that this induction is distinct from the induction on $\|\bk\|$ 
that we use to prove Lemma \ref{lem:property2'}). 
First, \eqn{eq:prop2'-end-pt} is true for $j = j_0$, by \eqn{eq:j_0}. 
Next, for $j > j_0$, suppose by induction hypothesis that 
$X^r_{\bk}((j-1)T) \geq 2c r^{s_2}$. 
Then, we consider two cases: (a) if $X^r_{\bk}((j-1)T) \geq 4c r^{s_2}$, 
then by statement (ii), \eqn{eq:prop2'-end-pt} holds for $j$; 
(b) if $X^r_{\bk}((j-1)T) \in \left[2cr^{s_2}, 4cr^{s_2}\right)$, 
then by statement (i), \eqn{eq:prop2'-end-pt} holds for $j$ as well. 
This completes the induction step and the proof of \eqn{eq:prop2'-end-pt} for 
all $j \in \left\{j_0, j_0 + 1, \cdots, \frac{r^{\alpha} +1}{T}\right\}$. 
Now by statement (ii), for each $j \in \left\{j_0 + 1, \cdots, \frac{r^{\alpha} +1}{T}\right\}$, 
\[
\inf_{t \in [(j-1)T, jT]} X^r_{\bk}(t) \geq \frac{1}{2}X^r_{\bk}((j-1)T) \geq \frac{1}{2} \cdot 2cr^s = cr^s,
\]
where the second inequality follows by \eqref{eq:prop2'-end-pt}. 
Thus, 
\[
\inf_{t \in [1, r^{\alpha}+1]} X^r_{\bk}(t) \geq \inf_{t \in [j_0 T, \frac{r^{\alpha}+1}{T}\cdot T]} X^r_{\bk}(t)
= \inf_{j \in \left\{j_0 + 1, \cdots, \frac{r^{\alpha} +1}{T}\right\}} \left(\inf_{t \in [(j-1)T, jT]}X^r_{\bk}(t) \right) 
\geq cr^s; 
\]
i.e., for all $t \in [1, r^{\alpha} + 1]$, 
\begin{equation}\label{eq:prop2'}
X^r_{\bk}(t) \geq cr^{s_2}. 
\end{equation}

In summary, we have established that with probability $1 - e^{-\text{poly}(r)}$, for all $t \in [1, r^{\alpha} + 1]$, 
\[
X^r_{\bk}(t) \geq c r^{s_2} = c r^{s(\bk)}. 
\]
By the stationarity of the processes $\bX^r(\cdot)$, we can conclude that 
\[
\pr\left(X^r_{\bk}(t) \geq c r^{s_2}, \ \forall t \in [0, r^{\alpha}]\right) \geq 1 - e^{-\text{poly}(r)}.
\]
This completes the induction step, assuming that Claim 1 holds. It is now left to prove Claim 1.

{\em Proof of Claim 1.} Let us first focus on the sub-interval $[0, T]$. 
Again, since the systems indexed by $r$ are in the stationary regime, 
probability bounds that we derive over the sub-interval $[0, T]$ 
extend naturally to other sub-intervals $[jT, (j+1)T]$, $j = 1, 2, \cdots$. 
Claim 1 is an easy consequence of the following claim. 

{\em Claim 2.} There exist positive constants $c$ and $\bar{c}$
such that with probability $1 - e^{-\text{poly}(r)}$, the following holds. 
\begin{itemize}
\item[(1)] If $X^r_{\bk}(0) \leq 4cr^{s_2}$, 
then 
\begin{equation}\label{eq:increase}
X^r_{\bk}(T) - X^r_{\bk}(0) \geq \bar{c}r^{2s_2 - p - \veps'}. 
\end{equation}
\item[(2)] If $X^r_{\bk}(0) \geq 2cr^{s_2}$, 
then 
\begin{equation}\label{eq:max-decrease}
\inf_{t \in [0, T]} X^r_{\bk}(t) \geq \frac{1}{2}X^r_{\bk}(0). 
\end{equation}
\end{itemize}
To see how Claim 2 implies Claim 1, recall that our systems are in the stationary regime. 
Thus, for any other sub-interval $[jT, (j+1)T]$, with the same probability $1 - e^{-\text{poly}(r)}$, 
statements (1) and (2) of Claim 2 hold, when we replace $0$ by $jT$ and $T$ by $(j+1)T$. 
Since there are a polynomial (in $r$) number of such sub-intervals, Claim 1 then follows immediately. 
We now prove Claim 2.

{\em Proof of Claim 2.} To bound $X^r_{\bk}(\cdot)$, we consider terms $S_1(T), S_2(T)$ and $S_3(T)$ defined in \eqn{eq:S1} -- \eqn{eq:S3} separately.

(a) First, consider the term $S_1(T)$ defined in \eqn{eq:S1}.
Focus on $A^r_{(\bk, \iota)}(T)$, the cumulative arrivals along the edge $(\bk, \iota)$, 
where we recall that $k_{\iota} \ge 1$, and $\tilde{\bk} = \bk - \be_{\iota}$. 
We have 
\[
A^r_{(\bk, \iota)}(T) = \Pi_{(\bk, \iota)}\left(\int_0^T \lambda_{\iota} r \frac{X^r_{\tilde{\bk}}(u)}{X^r_{(\iota)}(u)} du\right) 
\leq S_1(T).
\]
Let $F^r_1$ be the event defined by 
\[
F^r_1 = \left\{\frac{1}{2}r \leq Z^r(u) \leq \frac{3}{2}r, \text{ for all } u \in [0, T]\right\}.
\] 
Then, by inspecting the proof of the base case, 
it is easy to see that $\pr(F^r_1) \geq 1 - e^{-\text{poly}(r)}$. 
Furthermore, 
for sufficiently large $r$, under the event $F^r_1$, we have that for all $u \in [0, T]$, 
\[
X^r_{(\iota)}(u) \leq Z^r(u) + \left[Z^r(u)\right]^p \leq 2r.
\]
Thus, for sufficiently large $r$, 
under the event $E^r \cap F^r_1$, where we recall the definition of $E^r$ in \eqn{eq:event-ind-hypo}, 
we have that for all $u \in [0, T]$,
\[
\frac{X^r_{\tilde{\bk}}(u)}{X^r_{(\iota)}(u)} \geq \frac{\tilde{c}r^{s_1}}{2r} = \frac{\tilde{c}}{2}r^{s_1 - 1},
\]
from which it follows that
\[
\int_0^T \lambda_{\iota} r \frac{X^r_{\tilde{\bk}}(u)}{X^r_{(\iota)}(u)} du \geq 
\lambda_{\iota} r \cdot \frac{\tilde{c}}{2}r^{s_1 - 1}\cdot T 
= c_1 r^{s_1 + s_2 - 1 - \veps'}. 
\]
where $c_1 = \frac{1}{2}\lambda_{\iota} \tilde{c}$.
We also define event $G^r_1$ to be
\[
G^r_1 = \left\{\Pi_{(\bk, \iota)}\left(c_1 r^{s_1 + s_2 - 1 - \veps'}\right) \geq \frac{1}{2} c_1 r^{s_1 + s_2 - 1 - \veps'}\right\}.
\]
Then, $\pr(G^r_1) \geq 1 - e^{-\text{poly}(r)}$, 
and under the event $E^r \cap F^r_1 \cap G^r_1$, we have that
\begin{equation}
A^r_{(\bk,\iota)}(T) \geq \frac{1}{2} c_1 r^{s_1 + s_2 - 1 - \veps'}.
\end{equation}
Thus, 
\begin{equation}\label{eq:arrival-lb}
\pr\left(S_1 (T) \geq \frac{1}{2}c_1 r^{s_1 + s_2 - 1 - \veps'}\right) \geq \pr\left(A^r_{(\bk,\iota)}(T) \geq \frac{1}{2}c_1 r^{s_1 + s_2 - 1 - \veps'}\right) \geq 1 - e^{-\text{poly}(r)}.
\end{equation}
This completes our probability bound for $S_1(T)$ and part (a).

(b) Next, we consider the term $S_2(T)$ defined in \eqn{eq:S2}. 
We have 
\[
S_2(T) = \sum_{i: \bk+\be_i \in \ck} A^r_{(\bk+\be_i, i)}(T) = \sum_{i: \bk+\be_i \in \ck} 
\Pi_{(\bk + \be_i, i)}\left(\int_0^T \lambda_i r \cdot \frac{X^r_{\bk}(u)}{X^r_{(i)}(u)} du\right).
\]
Instead of deriving probability bounds for $S_2(T)$, in part (b) we will only derive 
a probability bound for 
\begin{equation}\label{eq:tilde-S-2}
\sum_{i: \bk+\be_i \in \ck} \int_0^T \lambda_i r \cdot \frac{X^r_{\bk}(u)}{X^r_{(i)}(u)} du.
\end{equation}
The reason is twofold. First, the term \eqn{eq:tilde-S-2} captures the order of magnitude of $S_2(T)$. 
Second, the bound that we will derive for the term \eqn{eq:tilde-S-2} 
depends on $X^r_{\bk}(0)$. Since we will consider two separate cases 
depending on the magnitude of $X^r_{\bk}(0)$, we will have separate probability bounds for $S_2(T)$, 
which we derive after considering the term $S_3(T)$.
Consider $\int_0^T \lambda_{i} r \frac{X^r_{\bk}(u)}{X^r_{(i)}(u)}du$, 
a generic summand of the term \eqn{eq:tilde-S-2}. 
By the base case and using the fact that $X^r_{(i)}(u) \geq X^r_{\bZero}(u)$ for all $u$, we have that
\begin{equation}\label{eq:lb1}
\pr\left(X^r_{(i)}(u) \geq \frac{1}{2} r^p\right) \geq 
\pr\left(X^r_{\bZero}(u) \geq \frac{1}{2} r^p\right) \geq 1 - e^{-\text{poly}(r)}.
\end{equation}
Furthermore, we claim that there exists a constant $c_2>0$ 
such that 
\begin{equation}\label{eq:lb2}
\pr\left(X^r_{\bk}(u) - X^r_{\bk}(0) \leq c_2rT, ~~ \forall u \in [0, T]\right) \geq 1 - e^{-\text{poly}(r)}.
\end{equation}
We establish \eqn{eq:lb2} as follows. 
Let $N^r$ be the total number of arrivals to and departures from the system up to time $T$. 
Then, it is clear that for all $u \in [0, T]$, 
$X^r_{\bk}(u) - X^r_{\bk}(0) \leq N^r$. We now obtain a probability bound on $N^r$. 
Since we are only interested in probability bounds, we 
can and will at different points of the proof use different underlying probability space constructions,
as long as they produce the same -- in law -- system process.
At this point, we will use the following, different probability space  construction.
We associate a driving unit-rate Poisson process $\Pi(\cdot)$ 
for all the arrivals to the system, and an independent unit-rate Poisson process $\wt{\Pi}(\cdot)$ 
to drive all the departures. The total arrival rate at all times is $r$, 
and the total departure rate at time $t$ is given by $\sum_i \mu_i Y^r_i(t)$. 
Thus, $N^r$ has the same distribution as $\Pi(rT) + \wt{\Pi}\left(\int_0^T \sum_i \mu_i Y^r_i(u) du\right)$. 
With probability $1 - e^{-\text{poly}(r)}$, $\Pi(rT) \leq 2rT$. 
By Lemma \ref{lem-p1}, it is easy to see that with probability $1 - e^{-\text{poly}(r)}$, for all $u \in [0, T]$, 
$\sum_i \mu_i Y^r_i(u) du \leq \sum_i \mu_i (2\rho_i r) = 2\sum_i \lambda_i r = c_2' r$ for $c_2' = 2\sum_i \lambda_i$.
This implies that with probability $1 - e^{-\text{poly}(r)}$, 
\[
\wt{\Pi}\left(\int_0^T \sum_i \mu_i Y^r_i(u) du\right) \leq \wt{\Pi}\left(\int_0^T c_2' r du\right) 
= \wt{\Pi}\left(c_2' r T\right) \leq 2c_2' r T.
\]
Thus, with probability $1 - e^{-\text{poly}(r)}$, for all $u \in [0, T]$, 
\[
X^r_{\bk}(u) - X^r_{\bk}(0) \leq N^r \stackrel{d}= \Pi(rT) + \wt{\Pi}\left(\int_0^T \sum_i \mu_i Y^r_i(u) du\right) \leq 2rT + 2c_2' r T = c_2 r T.
\]
where $c_2 = 2 + 2c_2'$. This establishes \eqn{eq:lb2}.

By \eqn{eq:lb1} and \eqn{eq:lb2}, we can bound \eqn{eq:tilde-S-2} 
as follows. With probability $1 - e^{-\text{poly}(r)}$, 
\begin{eqnarray}
\int_0^T \lambda_{i} r \frac{X^r_{\bk}(u)}{X^r_{(i)}(u)}du 
&\le & \int_0^T \lambda_{i} r \frac{X^r_{\bk}(0) + c_2 rT}{r^p/2}du \nonumber \\
&\le & c_3 T r^{1-p} X^r_{\bk}(0) + c_4 T^2 r^{2-p} \nonumber \\
&=& c_5 r^{s_2 -p - \veps'} \left(X^r_{\bk}(0) + r^{s_2 - \veps'}\right) \nonumber
\end{eqnarray}
for some positive constants $c_3$, $c_4$ and $c_5$. It then follows immediately that 
with probability $1 - e^{-\text{poly}(r)}$, 
\begin{equation}\label{eq:arrival-ub}
\sum_{i:\bk+\be_i \in \ck}\int_0^T \lambda_{i} r \frac{X^r_{\bk}(u)}{X^r_{(i')}(u)}du 
\leq c_6 r^{s_2 -p - \veps'} \left(X^r_{\bk}(0) + r^{s_2 - \veps'}\right),
\end{equation}
for some positive constant $c_6$. This completes part (b). 

(c) We now consider the term $S_3(T)$ defined in \eqn{eq:S3}, which is given by
\[
S_3(T) = \sum_{i: \bk-\be_i \in \bar{\ck}} D^r_{(\bk, i)} (T) = 
\sum_{i: \bk-\be_i \in \bar{\ck}} \wt{\Pi}_{(\bk, i)}\left(\int_0^T X^r_{\bk}(u) k_i \mu_i du\right).
\]
Similar to part (b), we only derive a probability bound for 
\begin{equation}\label{eq:tilde-S-3}
\sum_{i: \bk-\be_i \in \bar{\ck}} \int_0^T X^r_{\bk}(u) k_i \mu_i du.
\end{equation}
By \eqn{eq:lb2}, we have that with probability $1 - e^{-\text{poly}(r)}$, 
\begin{eqnarray}
\int_0^T k_i\mu_i X^r_{\bk}(u) du &\leq & \int_0^T k_i\mu_i \left(X^r_{\bk}(0) + c_2 rT\right) du \nonumber \\
&=& c_7 T X^r_{\bk}(0) + c_8 r T^2 \nonumber \\
& = & c_9 r^{s_2 - 1 - \veps'} \left(X^r_{\bk}(0) + r^{s_2 - \veps'}\right), \nonumber
\end{eqnarray}
and 
\begin{equation}\label{eq:departure-ub}
\sum_{i: \bk-\be_i \in \bar{\ck}}\int_0^T k_i\mu_i X^r_{\bk}(u) du \leq 
c_{10} r^{s_2 - 1 - \veps'} \left(X^r_{\bk}(0) + r^{s_2 - \veps'}\right), 
\end{equation}
for some positive constants $c_7, c_8, c_9$ and $c_{10}$. This completes part (c). 

Observe that by \eqn{eq:arrival-ub} and \eqn{eq:departure-ub}, 
we have that with probability $1 - e^{-\text{poly}(r)}$, 
\[
\sum_{i:\bk+\be_i \in \ck}\int_0^T \lambda_{i} r \frac{X^r_{\bk}(u)}{X^r_{(i)}(u)}du 
+ \sum_{i: \bk-\be_i \in \bar{\ck}}\int_0^T k_i\mu_i X^r_{\bk}(u) du
\leq c_{11} r^{s_2 -p - \veps'} \left(X^r_{\bk}(0) + r^{s_2 - \veps'}\right),
\]
for some positive constant $c_{11}$. 
For notational convenience, we use $\wt{S}$ to denote the LHS of the preceding inequality, 
i.e., 
\begin{equation}\label{eq:wt-S}
\wt{S} = \sum_{i:\bk+\be_i \in \ck}\int_0^T \lambda_{i} r \frac{X^r_{\bk}(u)}{X^r_{(i)}(u)}du 
+ \sum_{i: \bk-\be_i \in \bar{\ck}}\int_0^T k_i\mu_i X^r_{\bk}(u) du,
\end{equation}
and we have 
\begin{equation}\label{eq:arr-dep-ub}
\pr\left(\wt{S} \leq c_{11} r^{s_2 -p - \veps'} \left(X^r_{\bk}(0) + r^{s_2 - \veps'}\right) \right) \geq 1 - e^{-\text{poly}(r)}.
\end{equation}
Let us also recall $S_2(T)$ and $S_3(T)$ defined in \eqn{eq:S2} and \eqn{eq:S3}, 
and consider the distribution of $S_2(T) + S_3(T)$. 
Note that by our construction of the probability space, 
$S_2(T)$ and $S_3(T)$ are sums of terms that correspond to 
arrivals and departures driven by independent underlying Poisson processes. 
Since we are only interested in the distribution of $S_2(T) + S_3(T)$, 
at this point consider the following alternative construction of the probability space, 
where all arrivals and departures that appear in the summations of $S_2(T)$ 
and $S_3(T)$ are driven by a {\em common} unit-rate Poisson process $\Pi'(\cdot)$. 
Then $\wt{S}$ has the same distribution under the original and alternative constructions;
and $S_2(T) + S_3(T)$ under the original construction has the same distribution as 
 $\Pi'\left(\wt{S}\right)$
under the alternative construction.

We now complete the proof of Claim 2, making use of mainly \eqn{eq:arrival-lb} and \eqn{eq:arr-dep-ub}.

Let $c = \frac{c_1}{64c_{11}}$ and $\bar{c} = \frac{c_1}{4}$.
We first consider case 1 of Claim 2, and suppose that $X^r_{\bk}(0) \leq 4c r^{s_2}$. 
For sufficiently large $r$, $4c r^{s_2} \geq r^{s_2 - \veps'}$.
Thus, by \eqn{eq:arr-dep-ub}, 
with probability $1 - e^{-\text{poly}(r)}$, 
\[
\wt{S} \leq c_{11} r^{s_2 -p - \veps'} \left(X^r_{\bk}(0) + r^{s_2 - \veps'}\right)
\leq c_{11} r^{s_2 -p - \veps'} \cdot (8c r^{s_2})
\leq \frac{c_1}{8} r^{2s_2 - p - \veps'} = \frac{\bar{c}}{2}r^{2s_2 - p - \veps'}.
\]
Since $\veps' < s_2 - p/2$, $2s_2 - p - \veps' > 0$, and with probability $1 - e^{-\text{poly}(r)}$, we have 
\[
\Pi'(\wt{S}) \leq \Pi'\left(\frac{\bar{c}}{2} r^{2s_2 - p - \veps'}\right) \leq \bar{c} r^{2s_2 - p - \veps'}.
\]
Thus, 
the event
$$
\Pi'(\wt{S}) \leq \frac{c_1}{4} r^{2s_2 - p - \veps'},
$$
and then also the event
$$
S_2(T) + S_3(T)  \leq \frac{c_1}{4} r^{2s_2 - p - \veps'},
$$
hold with probability $1 - e^{-\text{poly}(r)}$.  By 
\eqn{eq:arrival-ub}, we have that with probability $1-e^{-\text{poly}(r)}$, 
\begin{eqnarray*}
X^r_{\bk}(T) - X^r_{\bk}(0) &\geq & S_1(T) - S_2(T) - S_3(T) \\
& \geq & \frac{c_1}{2} r^{s_1 + s_2 - 1 - \veps'} - \bar{c} r^{2s_2 - p - \veps'} \\
&=& \bar{c} r^{2s_2 - p - \veps'},
\end{eqnarray*}
where the last equality follows from the fact that $s_1 = s_2 + (1 - p)$. This completes the proof 
of case 1.

Next, consider case 2, and suppose that $X^r_{\bk}(0) \geq 2cr^{s_2}$. 
For sufficiently large $r$, $X^r_{\bk}(0) \geq 2cr^{s_2} \geq r^{s_2 - \veps'}$. 
Then, with probability $1-e^{-\text{poly}(r)}$, 
\[
\wt{S} \leq c_{11} r^{s_2 -p - \veps'} \left(X^r_{\bk}(0) + r^{s_2 - \veps'}\right) \leq 2c_{11} r^{s_2 -p - \veps'} X^r_{\bk}(0)
\leq \frac{1}{4}X^r_{\bk}(0),
\]
and with probability $1 - e^{-\text{poly}(r)}$, 
\[
\Pi'\left(\wt{S}\right) \leq \Pi'\left(\frac{1}{4}X^r_{\bk}(0)\right) \leq \frac{1}{2}X^r_{\bk}(0). 
\]
Thus, with probability $1 - e^{-\text{poly}(r)}$, for every $u \in [0, T]$,
\begin{eqnarray*}
X^r_{\bk}(u) - X^r_{\bk}(0) &\geq & - S_2(u) - S_3(u) \\
& \geq & -S_2(T) - S_3(T)\\
& \geq & -\frac{1}{2}X^r_{\bk}(0),
\end{eqnarray*}
where we note that the first two inequalities hold with probability $1$. 
This establishes \eqn{eq:max-decrease}, and completes the proof of case 2. 
This concludes the proof of Claim 2. 
\end{proof}

\end{document}